\documentclass[a4paper,10pt,reqno]{amsart}
\usepackage{amsmath}
  \usepackage{paralist}
  \usepackage{graphics} %% add this and next lines if pictures should be in esp format
  \usepackage{epsfig} %For pictures: screened artwork should be set up with an 85 or 100 line screen
\usepackage{graphicx}  \usepackage{epstopdf}%This is to transfer .eps figure to .pdf figure; please compile your paper using PDFLeTex or PDFTeXify.
 \usepackage[colorlinks=true]{hyperref}
\hypersetup{urlcolor=blue, citecolor=red}

\newtheorem{theorem}{Theorem}[section]

\newtheorem{lemma}[theorem]{Lemma}
\newtheorem{proposition}{Proposition}[section]

\theoremstyle{definition}
\newtheorem{definition}[theorem]{Definition}
\newtheorem{remark}{Remark}[section]

\newcommand{\spt}{\mathop\mathrm{spt}}
\renewcommand{\L}[1]{\mathbf{L^{#1}}}
\newcommand{\Lloc}[1]{\mathbf{L^{#1}_{loc}}}
\newcommand{\C}[1]{\mathbf{C^{#1}}}
\renewcommand{\d}{\mathrm{d}}
\newcommand{\R}{{\mathbb{R}}}
\newcommand{\N}{{\mathbb{N}}}
\usepackage[utf8]{inputenc}
\theoremstyle{definition}
\newtheorem{example}{Example}[section]
\newcommand{\Cc}[1]{\mathbf{C_c^{#1}}}

\newcommand{\tv}{\mathrm{TV}}
\newcommand{\BV}{\mathbf{BV}}
\usepackage{array}

\title[Particle approximation of the Hughes model]{Deterministic particle approximation of the Hughes model in one space dimension}

\author[M.~Di~Francesco and S.~Fagioli and M.~D.~Rosini and G.~Russo]{}


 \email{marco.difrancesco@univaq.it}
 \email{simone.fagioli@univaq.it}
 \email{mrosini@umcs.lublin.pl}
 \email{russo@dmi.unict.it}

\thanks{$^*$ Corresponding author: Marco Di~Francesco}

\begin{document}
\allowdisplaybreaks
\maketitle

% Enter the first author's name and address:
\centerline{\scshape Marco Di~Francesco$^*$and Simone Fagioli}
\medskip
{\footnotesize
% please put the address of the first author
\centerline{DISIM, Universit\`a degli Studi dell'Aquila,}
\centerline{via Vetoio 1 (Coppito), 67100 L’Aquila (AQ), Italy}
} % Do not forget to end the {\footnotesize by the sign }

\medskip

\centerline{\scshape Massimiliano~Daniele Rosini}
\medskip
{\footnotesize
 % please put the address of the second  and third author
\centerline{Instytut Matematyki, Uniwersytet Marii Curie-Sk\l odowskiej,}
\centerline{plac Marii Curie-Sk\l odowskiej 1, 20-031 Lublin, Poland}
}

\medskip

\centerline{\scshape Giovanni Russo}
\medskip
{\footnotesize
 % please put the address of the second  and third author
\centerline{Dipartimento di Matematica ed Informatica, Universit\`a di Catania,}
\centerline{Viale Andrea Doria 6, 95125 Catania, Italy}
}

\bigskip

\centerline{\emph{Dedicated to Professor Peter A. Markowich for his 60th birthday}}

\bigskip

% The name of the associate editor will be entered by an editorial staff
% "Communicated by the associate editor name" is not needed for special issue.
 \centerline{(Communicated by the associate editor name)}

%The abstract of your paper
\begin{abstract}
In this paper we present a new approach to the solution to a generalized version of Hughes' models for pedestrian movements based on a follow-the-leader many particle approximation. In particular, we provide a rigorous global existence result under a smallness assumption on the initial data ensuring that the trace of the solution along the turning curve is zero for all positive times. We also focus briefly on the approximation procedure for symmetric data and Riemann type data. Two different numerical approaches are adopted for the simulation of the model, namely the proposed particle method and a Godunov type scheme. Several numerical tests are presented, which are in agreement with the theoretical prediction.
\end{abstract}

%The title of your section 1
\section{Introduction}

%\rosso{I miei commenti appaiono in rosso. Una volta modificato l'articolo si possono buttare. Le parti aggiunte o modificate (a parte ovvie correzioni) sono in blu}
%
%\rosso{Ho cercato di uniformare l'uso del genitivo sassone, evidando ``the Hughes' model'' e lasciando {\em Hughes' model} o {\em the Hughes model}}

The mathematical modeling of human crowds has gained significant attention in the last decades, see \cite{bellomo_bellouquid_2011, ColomboRosini2009, Helbing2007, Henderson1971, Hughes2002, Muramatsu00, Piccoli2009, rosini_book, Treuille2006}. In 2002, Roger L.~Hughes \cite{Hughes2002} proposed a `thinking fluid' approach to this subject, which consists in modeling a human crowd as a continuum medium, with Eulerian velocity computed via a nonlocal constitutive law of the overall distribution of pedestrians. Such a nonlocal dependence is encoded in a \emph{weighted} distance function computed at a quasi-equilibrium regime, obeying an eikonal equation with right-hand side depending on the density of pedestrians. In its multidimensional formulation, the model reads
\begin{align}\label{eq:hughes_multid}
&\rho_t - \mathrm{div}\left(\rho \, v(\rho) \, \frac{\nabla \phi}{\|\nabla \phi\|}\right) = 0\,,
&\|\nabla \phi\| = c(\rho)\,.
\end{align}
Here $x$ denotes the position variable, $t$ is the time, and $\rho = \rho(t,x) \ge0$ represents the averaged crowd density. The model is typically  posed on a bounded domain $x\in \Omega\subset \R^2$, with boundary conditions depending on the position of the exits in $\partial\Omega$, see \cite{Hughes2002,BurgerDiFrancescoMarkowichWolfram} for further details. The map $\rho\mapsto v(\rho)$ represents the absolute value of the velocity, and is assumed to be decreasing, as lower velocities correspond to higher densities. The function $\phi$ in \eqref{eq:hughes_multid} models the pedestrians perception of the exits, and is computed as a weighted distance encoding the overall distribution of the crowd. It may be interpreted as an estimated exit time for a given distribution of pedestrians. More precisely, pedestrians choose their direction upon evaluation of a \emph{running cost} function $\rho\mapsto c(\rho)$. Such a function is assumed to be increasing, since pedestrians choose their path toward the closest exit avoiding densely crowded regions.

Hughes' approach addresses specifically the movement of pedestrians in highly populated domains. In \cite{Hughes2003}, the model has been suggested in particular to provide assistance in the annual Muslim Hajj and, quoting from \cite{Hughes2003}, `to locate barriers that actually increase the flow of pedestrians above that when there are no barriers present'. Such a phenomenon, analogous to a similar one known in vehicular traffic modeling, is often referred to in the literature as \emph{Braess' paradox} for pedestrians, see \cite{braess, rosini_book}. Hughes' model is way far from being the only significant PDE-based approach to pedestrian movements, see e.g.~\cite{ColomboRosini2009, maury2010}. Compared to other existing models in the literature, in addition to its potential applications listed in \cite{Hughes2003}, this model has the undeniable advantage of combining classical overcrowding avoidance effects with singular \emph{direction-switching} phenomena, which occur typically in a densely crowded region with limited information on the location of the exits. Clearly, the global perception of the domain implied by the eikonal equation in \eqref{eq:hughes_multid} is not totally realistic in most situations. However, the model is very flexible and can be easily improved to localize the range of perception, see \cite{carrillo_martin_wolfram}. Another argument in favour of \eqref{eq:hughes_multid} has been recently provided in \cite{BurgerDiFrancescoMarkowichWolfram}, in which a full-transient version of this model has been justified as a mean field game.

In this paper we focus on the Hughes model \cite{Hughes2002} in its one-dimensional formulation
\begin{subequations}\label{eq:model}
\begin{align}\label{eq:hughes_continuum}
&\rho_t - \left[\rho \, v(\rho) \, \frac{\phi_x}{|\phi_x|}\right]_x = 0,&
&|\phi_x|=c(\rho),&
x \in (-1,1),~ t > 0,
\end{align}
coupled with homogeneous Dirichlet boundary conditions
\begin{align}\label{eq:boundary_continuum}
&\rho(t,-1)=\rho(t,1)=0,&
&\phi(t,-1)=\phi(t,1)=0,&
t>0,
\end{align}
and initial datum
\begin{align}\label{eq:initial_continuum}
&\rho(0,x)=\bar{\rho}(x),&
&x \in (-1,1).
\end{align}
\end{subequations}
In this version of the model, two exits are located at the edges of the interval $(-1,1)$.

It will be assumed that a maximal value of $\rho$ is prescribed, which we set equal to $1$ for simplicity. A preliminary \emph{formal} maximum principle proved in \cite{El-KhatibGoatinRosini} shows that $\rho$ never exceeds the range $[0,1]$ if $\bar{\rho}$ takes values in $[0,1]$. We shall assume what follows.
\begin{itemize}
  \item [(V)] The velocity map $v \colon [0,1] \to [0,+\infty)$ is $\C1$, strictly decreasing, with $v(0) \doteq v_{\max}>0$ and $v(1)=0$. Moreover, there exists a $\hat\rho\in (0,1)$ such that $[v(\rho) + \rho \, v'(\rho)] (\hat\rho - \rho) > 0$ for all $\rho \in (0,1)\setminus\{\hat\rho\}$.
  \item [(C)] The running cost map $c \colon [0,1] \to [1,+\infty]$ is $\C2$, increasing, with $c''(\rho)>0$, and with $c(0)=1$\,.
\end{itemize}
An important example is
\[c(\rho)=\frac{1}{v(\rho)},\]
see \cite{Hughes2002,BurgerDiFrancescoMarkowichWolfram} and Example~\ref{exe:c} below.

The boundary conditions \eqref{eq:boundary_continuum} are justified as follows. If we denote the efficiency $\beta \in [0,1]$ of the exits by the ratio between the escape velocity allowed by the exit and the maximum pedestrian velocity, we have that the flux of pedestrians through the exit is proportional to $\beta$, i.e. we obtain the Robin type condition condition
\begin{equation}\label{eq:robin}
   \rho \, v(\rho)  = \beta \,\rho \, v_{\rm max},
\end{equation}
where the term $\phi_x/|\phi_x|$ gets cancelled when considering the outgoing normal direction of the flux. This is equivalent to imposing the exit velocity
\[
   v|_{x=\mp 1} = \beta \, v_{\rm max}.
\]
Since we assume that the velocity is related to the density via a monotonic function, then imposing the velocity is equivalent to imposing the density. Hence, at variance with what happens in diffusion problem in which Robin boundary conditions become in the limit either Dirichlet or Neumann condition, in our case we simply obtain a Dirichlet conditions on the density. For simplicity, we always considered the case of \emph{maximum} efficiency $\beta = 1$, i.e.\ $v = v_{\rm max}$, which in turn gives zero boundary conditions on the density.
Note however that in the cases in which the characteristics are exiting the domain, no boundary conditions have to be assigned on the boundary. A discussion about this issue is provided later on in this section. By a simple duality argument, a similar zero Dirichlet condition is obtained for $\phi$, see the discussion in \cite{BurgerDiFrancescoMarkowichWolfram}.

A fully satisfactory existence theory for \eqref{eq:model} is still missing. A mathematical theory in this setting was first addressed in \cite{DiFrancescoMarkowichPietschmannWolfram}, in which the eikonal equation $\|\nabla \phi\| = c(\rho)$ was replaced by two regularised versions involving a Laplacian term. In both cases, existence of weak solutions was proved, and the uniqueness of entropy solutions was provided in one of the two cases. \cite{DiFrancescoMarkowichPietschmannWolfram} was the first rigorous attempt at the mathematical theory for the Hughes model. The use of a regularised eikonal equation implies a $\C1$ regularity of the velocity field in the continuity equation for the density $\rho$. Such a property does not hold in the non-regularized case \eqref{eq:hughes_continuum}, and the numerical simulations performed in \cite{DiFrancescoMarkowichPietschmannWolfram} showed evidence of discontinuities in the velocity field, expressing sudden changes of direction in the movement of the pedestrians.

A first attempt at a rigorous mathematical treatment of \eqref{eq:hughes_continuum} (without any regularization) was performed independently in \cite{AmadoriDiFrancesco} and \cite{El-KhatibGoatinRosini}. In both papers, the authors provided a complete description of the behaviour of solutions for short times in the case of Riemann-type initial data on an interval. In \cite{AmadoriDiFrancesco} a solution semigroup was defined for short times for piecewise constant initial conditions. As a special case, solutions with direction-switching phenomena were produced analytically. In \cite{El-KhatibGoatinRosini} the Riemann problem was analyzed more thoroughly, and the optimization of the evacuation time in that case was studied. Numerical simulations were performed in \cite{DiFrancescoMarkowichPietschmannWolfram} and in \cite{El-KhatibGoatinRosini}, essentially based on Godunov schemes in both papers.

The results in \cite{AmadoriDiFrancesco} laid the basis for tackling the convergence of a wave-front-tracking scheme \cite{BressanLectureNotes, DafermosWFT} for \eqref{eq:model}, see the numerical simulations in \cite{GoatinMimault}. Such an approach was only partly successful, and led to the results in \cite{AmadoriGoatinRosini}, which show existence of entropy solutions when the initial condition is \emph{well-separated}, i.e.~is yielding the formation of two distinct groups of pedestrians moving towards the two exits, with the emergence of a vacuum region in between, persisting until the total evacuation of the domain. A more general set of initial conditions was never covered so far, the main difficulty lying in the very large number of possible case studies in the interactions between classical and nonclassical waves in the wave-front-tracking algorithm.

As already observed in \cite{AmadoriDiFrancesco, AmadoriGoatinRosini, El-KhatibGoatinRosini}, the one-dimensional Hughes model \eqref{eq:hughes_continuum} can be reformulated in terms of the density $\rho$ as follows:
\begin{align}\label{eq:hughes_reformulated}
\rho_t + F(t,x,\rho)_x = 0,&
&\int_{-1}^{\xi(t)}\!\!\!\! c(\rho(t,y)) \,\d y = \int_{\xi(t)}^1\!\!\!\! c(\rho(t,y)) \,\d y,&
&x\in (-1,1),~ t>0,
\end{align}
with
\begin{align*}
F(t,x,\rho) \doteq \mathrm{sign}(x-\xi(t)) \, f(\rho),&
&f(\rho) \doteq \rho \, v(\rho).
\end{align*}
Here, $\xi(t)$ plays the role of a \emph{turning point} at which pedestrians switch their target exit in view of the overall crowd distribution.
The reduction of \eqref{eq:hughes_continuum} to \eqref{eq:hughes_reformulated} is soon explained. In order to have a unique physically relevant (semi-concave) viscosity solution to the Dirichlet problem for $\phi$, the derivative $\phi_x$ can change its sign just once from positive to negative. The turning point $\xi(t)$ is then defined as the point $x$ at which $x\mapsto\phi(t,x)$ reaches its maximum. $\xi$ depends on time because of the $\rho$-depending term on the right hand side of the eikonal equation in \eqref{eq:hughes_continuum}. Indeed, $\xi(t)$ depends \emph{non-locally} on $\rho(t,\cdot)$. The second equation in \eqref{eq:hughes_reformulated} is recovered after a simple integration on the eikonal equation in \eqref{eq:hughes_continuum}, upon the assumption that $x\mapsto\phi(t,x)$ is continuous at $\xi(t)$. Notice that the flux $F$ is possibly discontinuous along the \emph{turning curve} $x = \xi(t)$. For a detailed study of a Rankine-Hugoniot condition along the turning point we refer to \cite{AmadoriDiFrancesco,El-KhatibGoatinRosini}.

\begin{example}
The simplest choice for the cost function is $c\equiv1$.
In this case the obvious solution to the eikonal equation in \eqref{eq:hughes_continuum} is the distance function $\phi(t,x)=\mathrm{dist}(x,\{-1,1\})$ for all $t\geq 0$.
As a consequence, the ratio $\phi_x/|\phi_x|$ is equal to $-1$ for $x<0$ and to $1$ for $x>0$, and the model is equivalent to the scalar conservation law with discontinuous flux
\[
\rho_t + \left[\vphantom{\sum} \mathrm{sign}(x) \, \rho \, v(\rho)\right]_x=0.
\]
As a result, Hughes' model reduces to the Lighthill-Whitham-Richards (LWR) model for traffic flow \cite{LWR1, Richards} with negative velocity on $(-1,0)$ and positive velocity on $(0,1)$.
%In particular, the unique turning curve is stationary and is located in $x=0$.
Roughly speaking, pedestrians are moving toward the closest exit regardless of the overall distribution. This is the typical behaviour of a crowd in a \emph{panic} state, see \cite{El-KhatibGoatinRosini}.
\end{example}

\begin{example}\label{exe:c}
In the literature, the usual choice for the cost function is
\begin{equation}\label{eq:originalc}
c(\rho) = \frac{1}{v(\rho)}
\end{equation}
see \cite{AmadoriDiFrancesco, AmadoriGoatinRosini, BurgerDiFrancescoMarkowichWolfram, DiFrancescoMarkowichPietschmannWolfram, Hughes2002, Hughes2003, TwarogowskaGoatinDuvigneau}.
In this case, in order to bypass the technical issue of $c$ blowing up at $\rho=1$ in the eikonal equation of \eqref{eq:hughes_continuum}, the initial condition $\bar{\rho}$ is assumed to take values in $[0, 1-\delta]$, for some constant $\delta>0$ sufficiently small.
This assumption, together with the maximum principle, ensures that the cost \eqref{eq:originalc} computed along any solution of \eqref{eq:model} is well defined.
\end{example}

Rewriting Hughes' model in the form \eqref{eq:hughes_reformulated} highlights the non-local mechanism by which the overall distribution of pedestrians in the domain biases the choice of the exit path for each single pedestrian. Since $c$ is increasing w.r.t.~$\rho$, such a mechanism penalises regions with high density in the choice of the nearest exit. In the relevant case in \eqref{eq:originalc}, $c(\rho)$ gets very large when $\rho$ is close to $1$, and the weighted distance function $\phi$ gets very steep in this case.

The reformulation of \eqref{eq:hughes_continuum} in the form \eqref{eq:hughes_reformulated} clearly suggests that Hughes' model can be seen as a two-sided first order LWR model for traffic flow \cite{LWR1, Richards}, with the turning point $\xi(t)$ splitting the whole interval $(-1,1)$ into two subintervals. Hinted by the new particle approach developed in \cite{DF_rosini}, in which nonnegative entropy solutions to nonlinear scalar conservation laws $\rho_t+[\rho\, v(\rho)]_x = 0$ are rigorously approximated by the empirical measure of the follow-the-leader ODE particle system, in which particle $i$ satisfies an equation of the form
\begin{equation*}
  \dot{x}_i = v\left(\frac{m}{x_{i+1}-x_i}\right),
\end{equation*}
where $m$ denotes the mass assigned to each particle, in the present paper we propose a new \emph{deterministic particle} approach to \eqref{eq:hughes_reformulated}, which can be roughly formulated as follows. Here we just outline our particle method and refer to section \ref{sec:particle} for its detailed application to approximating the solution of \eqref{eq:hughes_reformulated} via atomization.

For a fixed $N\in \mathbb{N}$ and a fixed mass $m>0$, let $-1<\bar{x}_0<\ldots<\bar{x}_N<1$ be the initial positions of $N+1$ particles in $(-1,1)$. The initial position $\bar{\xi}$ of the discrete turning point is determined uniquely by the identity
\begin{equation}
\begin{split}
& \left [\bar{x}_0+1\right ]+\sum_{i=0}^{I_0-1} \left [\bar{x}_{i+1}-\bar{x}_i\right ] c\left(\frac{m}{\bar{x}_{i+1}-\bar{x}_i}\right)+[\bar\xi-\bar{x}_{I_0}]\\
=&\left [\bar{x}_{I_0+1}-\bar\xi\right ]+\sum_{i=I_0+1}^{N-1} \left [\bar{x}_{i+1}-\bar{x}_i\right ] c\left(\frac{m}{\bar{x}_{i+1}-\bar{x}_i}\right)+\left [1-\bar{x}_N\right ],\label{eq:initial_turning_discrete_intro}
\end{split}
\end{equation}
which replaces the integral identity in \eqref{eq:hughes_reformulated}, with the density $\rho$ replaced by a discretized particle density, and with the assumption that a vacuum region is placed around $\bar\xi$ and near the boundary. We observe that here we assume zero density around the turning point and outside the interval $[\bar{x}_0,\bar{x}_N]$, which yields \eqref{eq:initial_turning_discrete_intro} in view of the assumption  $c(0)=1$. The pair $(I_0,\bar\xi)$ is uniquely determined by condition (\ref{eq:initial_turning_discrete_intro}) and
$\bar{x}_{I_0}\leq\bar\xi<\bar{x}_{I_0+1}$.

Assuming for simplicity that $\bar{x}_{I_0}\neq\bar\xi$, i.e.\ that $\bar{x}_{I_0}<\bar\xi<\bar{x}_{I_0+1}$, we let the particles $\bar{x}_0,\ldots,\bar{x}_{I_0}$ move towards the left exit via a backward follow-the-leader system, and the particles $\bar{x}_{I_0+1},\ldots,\bar{x}_N$ move towards the right exit via a forward one. More precisely,
\begin{subequations}\label{eq:discrete_Hughes}
\begin{equation}\label{eq:ftl_intro}
\begin{cases}
\dot{x}_0 = -v(0)\\
\dot{x}_i = -v\left(\frac{m}{x_{i}-x_{i-1}}\right),&i=1,\ldots,I_0,\\
\dot{x}_i = v\left(\frac{m}{x_{i+1}-x_i}\right),&i=I_0+1,\ldots,N-1,\\
\dot{x}_N=v(0)\,,
\end{cases}
\end{equation}
with $x_i(0)=\bar{x}_i$ for all $i=0,\ldots,N$. The discrete turning point $\xi(t)$ is required to obey $\xi(0)=\bar\xi$ and the identity
\begin{equation}\label{eq:turning_discrete_intro}
\begin{split}
& \left [x_0+1\right ]+\sum_{i=1}^{I_0-1} \left [x_{i+1}-x_i\right ]c\left(\frac{m}{x_{i+1}-x_i}\right)+\left [\xi-x_{I_0}\right ]\\
=&\left [x_{I_0+1}-\xi\right ]+\sum_{i=I_0+1}^{N-1} \left [x_{i+1}-x_i\right] c\left(\frac{m}{x_{i+1}-x_i}\right)+\left [1-x_N\right ]\,.
\end{split}
\end{equation}
\end{subequations}
%\begin{equation}\label{eq:turning_discrete_intro}
%\begin{array}{r@{~}*1{>{\displaystyle}c}p{}}
%& \left [x_1(t)+1)\right ]+\sum_{i=1}^{I-1} \left [x_{i+1}(t)-x_i(t)\right ]c\left(\frac{m}{x_{i+1}(t)-x_i(t)}\right)+\left [\xi(t)-x_I(t)\right ]\\
%  & \ = \left [x_{I+1}(t)-\xi(t)\right ]+\sum_{i=I+1}^{N-1} \left [x_{i+1}(t)-x_i(t)\right] c\left(\frac{m}{x_{i+1}(t)-x_i(t)}\right)+\left [1-x_N(t)\right ],
%\end{array}
%\end{equation}
The above formula is used as long as all points $\{x_i\}_{i=0}^N$ belong to the interval $[-1,1]$, and is suitably modified by removing the particles leaving the domain and adjusting the length of the first and last interval
%with the requirement that particles \emph{leaving} the interval $[-1,1]$ are \emph{removed} from formula \eqref{eq:turning_discrete_intro}
(see the subsection~\ref{sec:particle} for the details). Clearly, system \eqref{eq:ftl_intro} must be restarted once $\xi(t)$ crosses one of the two neighbor particles $x_{I_0}(t)$, $x_{I_0+1}(t)$. In this case, a detailed analysis comparing the speed of the neighbor particle with $\dot{\xi}$ at the crossing time is needed in order to define the post-crossing evolution consistently. We shall refer to system \eqref{eq:discrete_Hughes} as the \emph{FTL-Hughes model}.

The goal of this paper is twofold.
\begin{itemize}
  \item First of all we use the FTL-Hughes model \eqref{eq:discrete_Hughes} as a way to prove existence of weak solutions to the continuum model \eqref{eq:model}. We perform this task under analogous assumptions to those required in \cite{AmadoriGoatinRosini}, namely in situations in which a vacuum region is generated around $x=\xi(t)$, but with a lighter proof and under more general assumptions on the cost function $c$ (only the cost function $c(\rho)=1/v(\rho)$ is covered in \cite{AmadoriGoatinRosini}). As a byproduct of our result, we obtain a rigorous deterministic many-particle approximation result for a class of solutions to \eqref{eq:model}.
  \item As a second task, we implement the FTL-Hughes scheme \eqref{eq:discrete_Hughes} numerically in more complicated situations yielding particles crossing the turning point. We shall compare these simulations to existing ones in the literature \cite{DiFrancescoMarkowichPietschmannWolfram, El-KhatibGoatinRosini} based on Godunov's scheme, and show that our deterministic particle approach is able to capture mass transfer phenomena around the turning point.
\end{itemize}
Next we provide our notion of solution to \eqref{eq:model} and state our main results. For notational simplicity, we introduce the maximal initial density $\rho_{\max} \in (0,1]$ such that $c$ is well defined on $[0,\rho_{\max}]$. Then the initial datum $\bar{\rho}$ is assumed to be in $\L\infty((-1,1);\R)$ with $\|\bar{\rho}\|_\infty \le \rho_{\max}$.

Since we are dealing with zero Dirichlet boundary data, assuming strict concavity of $\rho\mapsto f(\rho)=\rho\, v(\rho)$ on $[0,1]$, it is easily seen that the proper boundary condition set in \cite{bardos} is $\mathrm{Tr}(\rho)\in [0,\hat{\rho}]$, where $\hat{\rho}\in (0,1)$ is uniquely determined by $f'(\hat{\rho})=0$, see condition (V) above. Such a condition expresses the fact that characteristics are pointing outward the domain (both in $x=-1$ and in $x=1$) if and only if $\mathrm{Tr}(\rho)\in [0,\hat{\rho}]$ on those two points. It is well known (see \cite{Dubois1988}) that the boundary problem with Dirichlet conditions for a nonlinear conservation law is solved by considering a Riemann problem involving the boundary datum and the trace $\mathrm{Tr}(\rho)$ at the boundary. Since the boundary datum is zero and since the flux is concave, the Riemann solver at the boundary only consists of rarefaction waves, which leave the domain if $\mathrm{Tr}(\rho)<\hat{\rho}$ and are centered on the value $\rho=\hat{\rho}$ at the boundary points if $\mathrm{Tr}(\rho)\geq \hat{\rho}$. Now, a solution of the Cauchy problem with a general $\L\infty$ initial datum with support on $[-1,1]$ and \emph{without boundary conditions} within the wave-front-tracking algorithm \cite{BressanLectureNotes, DafermosWFT} easily shows that $\rho$ achieves only values in $[0,\hat{\rho})$ outside $[-1,1]$. Therefore, the solution via wave-front-tracking without boundary data coincides with the one in which the boundary datum $\rho=0$ is imposed every time a wave interacts with the boundary. Hence, by using the convergence results in ~\cite{karlsen_lie_risebro_1999}, we recover that actually \emph{no boundary conditions have to be prescribed in our case}, and the solution $\rho$ to \eqref{eq:model} can be constructed by coupling the integral relation in \eqref{eq:hughes_reformulated} with the continuity equation on the \emph{whole} $\R$ restricted to the interval $[-1,1]$.

For this reason, also due to the fact that we shall only consider solutions with a vacuum region around the turning point, our notion of entropy solution will simply be obtained by `merging' the two traffic equations on $[-1,\xi(t)]$ and $[\xi(t),1]$ respectively. Observe that the strong traces of the solution at the boundary points exist due to the genuine non-linearity of the flux
(\cite{Panov, Vasseur}) and must satisfy
\begin{align*}
&\rho(t,-1^+)\leq \hat{\rho},& \rho(t,1^-)\leq \hat{\rho}\,.
\end{align*}

%\textcolor{red}{In this paper we only deal with boundary conditions with infinite efficiency $\beta$ in \eqref{eq:robin}. We shall address the case with finite $\beta$ in a future work. The latter problem, which is more realistic in most situations, can be formulated via a point constraint on the flow at the two boundary points, see ...}
In this paper we only deal with boundary conditions with maximal efficiency $\beta=1$ in \eqref{eq:robin}. We shall address the case $\beta<1$ in a future work. We also plan to impose (local and non-local) point constraints on the flow at the two boundary points to model the presence, for instance, of exit doors, see \cite{AndreianovDonadelloRazafisonRosini2015-2, AndreianovDonadelloRosini2014, AndreianovDonadelloRazafisonRosini2015-1, ColomboRosini2009, ColomboGoatinRosini2010, ColomboRosini2005} or \cite[Chapter~6]{rosini_book} for more details.

We now state our notion of solution, which is motivated by the presence of a persistent vacuum region around the turning point for the class of solutions considered in this paper.
\begin{definition}\label{def:entropy_solution}
A map $\rho\in \L\infty([0,+\infty)\times \R;[0,\rho_{\max}])$ is called a (well-separated) \emph{entropy solution} to \eqref{eq:model} if
\begin{itemize}
\item There exists $\varepsilon>0$ such that $\rho$ is equal to zero on the open cone \[\mathcal{C}\doteq\left \{(t,x) \in \R_+ \times \R \colon |x-\bar\xi| < \varepsilon \, t\right \}.\]
  \item $\rho\,\mathbf{1}_{(-\infty,\bar\xi)}$ is an entropy solution in the Kruzhkov sense \cite{Kruzhkov} to the Cauchy problem for $\rho_t -[\rho\, v(\rho)]_x = 0$ with initial datum $\bar\rho\,\mathbf{1}_{(-\infty,\bar\xi)}$.
  \item $\rho\,\mathbf{1}_{(\bar\xi,+\infty)}$ is an entropy solution in the Kruzhkov sense \cite{Kruzhkov} to the Cauchy problem for $\rho_t +[\rho\, v(\rho)]_x = 0$ with initial datum $\bar\rho\,\mathbf{1}_{(\bar\xi,+\infty)}$.
  \item The turning curve $\mathcal{T}\doteq\left \{(t,x) \in \R_+ \times [-1,1] \colon x=\xi(t)\right \}$ is continuous with $\xi(0)=\bar{\xi}$. Moreover, $\rho(t,\cdot)$ and $\xi(t)$ satisfy
      \begin{equation}\label{eq:integral_equation}
        \int_{-1}^{\xi(t)}c(\rho(t,y)) \, {\d} y = \int_{\xi(t)}^1 c(\rho(t,y)) \, {\d} y\,,
      \end{equation}
  for almost every $t\geq 0$\,. Finally, $\mathcal{T}\subset \mathcal{C}$.
\end{itemize}
\end{definition}

As a first result, we provide the following theorem for the `symmetric' case. Let us denote by $\mathcal{S}$ the space of functions $\rho\in\L\infty\left((-1,1);[0,1]\right)$ with $\|\rho\|_\infty\le\rho_{\max}$ that are \emph{even}, namely $\rho(x) = \rho(-x)$ for a.e.~$x\in(-1,1)$.

\begin{theorem}\label{teo:1}
For any initial datum $\bar\rho$ in $\mathcal{S}$, there exists a unique entropy solution $\rho$ to \eqref{eq:model} in the sense of Definition~\ref{def:entropy_solution}, such that $\rho(t,\cdot) \in \mathcal{S}$ for all $t>0$. Such a solution is obtained as a strong $\L1$ limit of the discrete density $R(t,x)$ constructed in subsection~\ref{sec:particle} via the FTL-Hughes particle system \eqref{eq:discrete_Hughes}.
\end{theorem}

In the next theorem we state our main result, which deals with a class of \emph{small} initial data in $\BV$. For further use, we define the function
\begin{equation}\label{eq:upsilon}
\Upsilon(\rho) \doteq c(\rho) - c'(\rho) \, \rho\,,
\end{equation}
which is strictly decreasing in view of assumption (C) above. We then set
\begin{align}
 & L \doteq \mathrm{Lip}[\Upsilon|_{[0,\rho_{\max}]}]=\max \left\{  c''\left(\rho\right) \rho \colon \rho \in \left[0,\rho_{\max}\right] \right\}\,,\label{eq:lipschitz}\\
 & C \doteq c'(\rho_{\max})\, \rho_{\max} = \max\left\{c'(\rho) \, \rho \colon \rho \in \left[0,\rho_{\max}\right]\right\}\,.\label{eq:Cgrande}
\end{align}

\begin{theorem}\label{teo:2}
If $\rho_{\max}<1$ and the initial datum $\bar\rho\in \BV((-1,1);\left[0,\rho_{\max}\right])$ satisfies
\begin{equation}\label{eq:inequality}
  \frac{v_{\max}}{2}
\left[\vphantom{\sum} L\,{\tv}(\bar{\rho})
+
3 \, C\right]
<
v(\rho_{\max}),
\end{equation}
then there exists a unique entropy solution $\rho$ to \eqref{eq:model} in the sense of Definition~\ref{def:entropy_solution} defined globally in time. Such a solution is obtained as a strong $\L1$ limit of the discrete density $R(t,x)$ constructed in subsection~\ref{sec:particle} via the FTL-Hughes particle system \eqref{eq:discrete_Hughes}.
\end{theorem}

We remark that the $\rho_{\max}<1$ assumption above is essential in order to have the right-hand-side in the inequality \eqref{eq:inequality} strictly positive.

The proofs of Theorems \ref{teo:1} and \ref{teo:2} are carried out in subsections~\ref{sec:mainproof1} and \ref{sec:mainproof2} respectively.

The paper is structured as follows. In section~\ref{sec:results} we cover the analytical theory. In particular, in subsection~\ref{sec:particle} we set up the deterministic particle scheme for \eqref{eq:model}, namely we construct the approximation to \eqref{eq:hughes_reformulated} via the FTL-Hughes model \eqref{eq:discrete_Hughes}. In subsections~\ref{sec:mainproof1} and \ref{sec:mainproof2} we prove our analytical results, with Theorem~\ref{teo:2} being our main result. In section~\ref{sec:numerics} we show the numerical simulations of our particle methods in simple Riemann-type initial conditions. In particular, we compare our tests with previous ones performed in \cite{DiFrancescoMarkowichPietschmannWolfram, GoatinMimault}. The numerical section shows evidence that the two methods, though conceptually different, are in very good agreement when using a large number of particles. We stress here that, although the analytical results concerning our deterministic particle method are restricted to cases in which particle separate into two distinct sets keeping the same direction for all times, the numerical simulations also cover cases with direction switching, showing that the particle approach works also in those cases. Our study is restricted to the one-dimensional case. Possible extensions to multidimensional cases rely on the extension of the results in \cite{DF_rosini} to multi dimensional scalar conservation laws, which is not available in the literature.

\section{Analytical results}\label{sec:results}

\subsection{Particle approximation}\label{sec:particle}

Let $\bar{\rho}$ be in $\L\infty((-1,1);[0,\rho_{\max}])$. For a fixed integer $n\in \N$, we set $N \doteq 2^n$ and $m \doteq 2^{-n} \, M$, where $M\doteq\|\bar{\rho}\|_1$.
We set
\[
\bar{x}_0 \doteq \min\left\{\spt\bar{\rho}\right\},
\]
where $\spt$ denotes the support. We recursively define
\begin{align*}
&\bar{x}_i \doteq \inf\left\{ x > \bar{x}_{i-1} \colon \int_{\bar{x}_{i-1}}^x \bar{\rho}(y) \, \d y \geq m\right\},&
&i\in\left\{1,\ldots,N\right\}.
\end{align*}
The above defines a set of $N + 1$ particles $-1\le\bar{x}_0<\bar{x}_1<\ldots<\bar{x}_{N-1}<\bar{x}_N\le1$ with the property
\begin{align*}
&\int_{\bar{x}_i}^{\bar{x}_{i+1}} \bar{\rho}(y) \,\d y = m,&
&i\in\left\{0,\ldots,N-1\right\}.
\end{align*}
In particular, this implies that
\begin{align*}
&\frac{m}{\rho_{\max}} =
\frac{1}{\rho_{\max}}\int_{\bar{x}_i}^{\bar{x}_{i+1}} \bar{\rho}(y) \,\d y \le
\bar{x}_{i+1} - \bar{x}_i,&
&i\in\left\{0,\ldots,N-1\right\}.
\end{align*}
For future reference, we denote the local discrete initial densities
\begin{align*}
&\bar{R}_{i+\frac{1}{2}} \doteq \frac{m}{\bar{x}_{i+1}-\bar{x}_i},&
&i\in\left\{0,\ldots,N-1\right\},
\end{align*}
and introduce the discretized initial density $\bar{R} \colon \R \to [0,\rho_{\max}]$ by
\begin{equation}
\bar{R}(x) \doteq \sum_{i=0}^{N-1} \bar{R}_{i+\frac{1}{2}} \, \mathbf{1}_{[\bar{x}_i,\bar{x}_{i+1})}(x).
\label{eq:R_bar}
\end{equation}
We now define the initial approximated turning point $\bar{\xi}^n$ via the formula
\begin{equation}
\int_{-1}^{\bar{\xi}^n} c\left(\bar{R}(y)\right) \,\d y =
\int_{\bar{\xi}^n}^1 c\left(\bar{R}(y)\right) \,\d y.
\label{eq:TP}
\end{equation}
Clearly, the above formula defines $\bar{\xi}^n \in (-1,1)$ uniquely.

The next step is the definition of the evolving particle scheme.
Roughly speaking, $\bar{\xi}^n$ splits the set of particles into left and right particles, the former moving according to a \emph{backward} follow-the-leader scheme, the latter according to a \emph{forward} one.
%%The use of the follow-the-leader scheme is motivated by \cite{DFRosini}.
%We have the following two cases: either $\bar{\xi}^n$ does not coincide with any of the $\bar{x}_i$ points, or $\bar{\xi}^n \in \{\bar{x}_1,\ldots,\bar{x}_{N-1}\}$.
%% (we clearly exclude the possibility that the turning point is located at the boundary).
By a slight modification of the initial condition, we may always assume that $\bar{\xi}^n$ does not coincide with any of the particles.
Then, there exists $I_0\in \{0,\ldots,N\}$ such that $\bar{\xi}^n \in (\bar{x}_{I_0},\bar{x}_{I_0+1})$.
We then set
\begin{equation}\label{eq:FTL}
\begin{cases}
\dot{x}_0(t)= -v_{\max},\\
\dot{x}_i(t)=-v\left(\frac{m}{x_{i}(t)-x_{i-1}(t)}\right), &i \in \{1,\ldots, I_0\},\\
\dot{x}_i(t)=v\left(\frac{m}{x_{i+1}(t)-x_i(t)}\right), &i \in \{I_0+1,\ldots, N-1\},\\
\dot{x}_N(t)= v_{\max},\\
x_i(0)=\bar{x}_i,&i \in \{0,\ldots, N\}.
\end{cases}
\end{equation}
We consider the corresponding discrete density
\begin{align*}
&R_{i+\frac{1}{2}}(t) \doteq \frac{m}{x_{i+1}(t)-x_i(t)},&
&i \in \{0,\ldots,N-1\}\setminus\{I_0\},
\\
&R_{i+\frac{1}{2}}(t) \doteq 0,&
&i \in \{-1,I_0,N\}.
\end{align*}
Notice that the above density has been set to equal zero outside the particle region $[x_0(t),x_N(t))$ and around the turning point, namely in $[x_{I_0}(t),x_{I_0+1}(t))$. The latter in particular is simply due to a consistency with the numerical simulations, in which the computation of the turning point is made simpler in this way. This simplifying assumption will introduce a small error $m=2^{-n}M$ in the total mass.

In view of the above notation, \eqref{eq:FTL} can be re-written as
\begin{equation}\label{eq:DevinTownsend2}
\begin{cases}
\dot{x}_i(t)=-v\left(R_{i-\frac{1}{2}}(t)\right),
&i \in \{0,\ldots, I_0\},
\\
\dot{x}_i(t)=v\left(R_{i+\frac{1}{2}}(t)\right),
&i \in \{I_0+1,\ldots, N\},\\
x_i(0)=\bar{x}_i,&i \in \{0,\ldots, N\}.
\end{cases}
\end{equation}
Notice that $R_{I_0+1/2}$ does not bias the movement of any of the particles, and this is an argument in favour of the ansatz $R_{I_0+1/2}(t) = 0$. For future reference, we compute for any $i \in \{0,\ldots, N-1\} \setminus\{I_0\}$
\begin{equation}\label{eq:DevinTownsend}
\dot{R}_{i+\frac{1}{2}}(t) = -\frac{m \left[\dot{x}_{i+1}(t)-\dot{x}_i(t)\right]}{\left[x_{i+1}(t)-x_i(t)\right]^2}
= -R_{i+\frac{1}{2}}(t) \, \frac{\dot{x}_{i+1}(t)-\dot{x}_{i}(t)}{x_{i+1}(t)-x_{i}(t)}.
\end{equation}
In particular, the time derivative of the discrete density $R_{i+1/2}$ satisfies
\begin{equation}
\begin{cases}
\dot{R}_{i+\frac{1}{2}}(t) = \dfrac{m\left[v(R_{i+\frac{1}{2}}(t))-v(R_{i-\frac{1}{2}}(t))\right]}{[x_{i+1}(t)-x_i(t)]^2},
&i \in \{0,\ldots, I_0-1\},\\[10pt]
\dot{R}_{i+\frac{1}{2}}(t) = - \dfrac{m\left[v(R_{i+\frac{3}{2}}(t))-v(R_{i+\frac{1}{2}}(t))\right]}{\left[x_{i+1}(t)-x_i(t)\right]^2},
&i \in \{I_0+1,\ldots, N-1\}.
\end{cases}
\label{eq:discrete_continuity_equation}
\end{equation}

The (unique) solution to the system \eqref{eq:FTL} is well defined until the turning point does not collide with a particle.
Similarly, the expressions \eqref{eq:discrete_continuity_equation} for the time derivative of the discrete densities only hold until the first collision of a particle with the turning point.
We note that the density $R_{I_0+1/2}(t)$ is equal to zero until the turning point collides with a particle.

In view of the discussion before Definition~\ref{def:entropy_solution} on the solution of the continuum problem at the boundary, we shall not impose any boundary condition to the particle system \eqref{eq:FTL}, and we shall follow the movement of each particle whether or not they are in $[-1,1]$. Hence, the discrete densities $R_{i+1/2}(t)$ are (in principle) defined for all $t>0$ and for all $i \in \{0,\ldots,N-1\}$. We remark that a careful analysis of the many particle approximation for the Dirichlet boundary value problem for a scalar conservation law is performed in \cite{DiFrancescoFagioliRosiniRusso2}.

The approximated turning point $\xi^n(t)$ is implicitly uniquely defined by
\begin{equation}\label{eq:equilibriocosto}
\int_{-1}^{\xi^n(t)} c(R(t,y)) \, \d y = \int^1_{\xi^n(t)} c(R(t,y)) \, \d y,
\end{equation}
where $R \colon \R_+\times\R \to [0,\rho_{\max}]$ is the discretized density defined by
\begin{equation}\label{eq:density}
  R(t,x) \doteq \sum_{i=0}^{N-1} R_{i+\frac{1}{2}}(t) \, \mathbf{1}_{[x_i(t),x_{i+1}(t))}(x).
\end{equation}
Clearly $\xi^n(t)$ belongs to $(-1,1)$ for any $t\ge0$. We underline that $\xi^n(0)$ does not necessarily coincide with $\bar{\xi}^n$.
We compute now the time derivative of $\xi^n(t)$.

\begin{lemma}\label{lem:dotxi}
Fix $t>0$, and let $I_-,I_0,I_+ \in \{1,\ldots,N-1\}$ be such that
\begin{align*}
&-1 \in \left[x_{I_--1}(t),x_{I_-}(t)\right),&
&\xi^n(t) \in \left(x_{I_0}(t),x_{I_0+1}(t)\right),&
&1 \in \left(x_{I_+}(t) ,x_{I_++1}(t)\right].
\end{align*}
At time $t$ we have then
\begin{align*}
2 \, \dot{\xi}^n=&
\left\{
\left[v(R_{I_--\frac{3}{2}})-v(R_{I_--\frac{1}{2}})\right]
\, \frac{x_{I_-}+1}{x_{I_-}-x_{I_--1}}
+v(R_{I_--\frac{1}{2}})
\right\}
c'(R_{I_--\frac{1}{2}}) \, R_{I_--\frac{1}{2}}
\\
&-v(R_{I_0-\frac{1}{2}}) \left[1-\Upsilon(R_{I_0-\frac{1}{2}})\right]
+\sum_{i=I_-}^{I_0-1} v(R_{i-\frac{1}{2}})
\left[ \Upsilon(R_{i-\frac{1}{2}}) - \Upsilon(R_{i+\frac{1}{2}})\right]\\
&+v(R_{I_0+\frac{3}{2}}) \left[1-\Upsilon(R_{I_0+\frac{3}{2}})\right]
+\sum_{i=I_0+2}^{I_+} v(R_{i+\frac{1}{2}})  \left[\Upsilon(R_{i-\frac{1}{2}}) - \Upsilon(R_{i+\frac{1}{2}})\right]\\
&-\left\{ v(R_{I_++\frac{1}{2}}) +\left[\vphantom{x_{I_-}} v(R_{I_++\frac{3}{2}})-v(R_{I_++\frac{1}{2}})\right] \, \frac{1-x_{I_+}}{x_{I_++1}-x_{I_+}} \right\}
c'(R_{I_++\frac{1}{2}}) \, R_{I_++\frac{1}{2}} ,
\end{align*}
where $\Upsilon$ is defined in \eqref{eq:upsilon}.
\end{lemma}
\begin{proof}
By \eqref{eq:equilibriocosto} we have that
\[
\begin{split}
&
\left[x_{I_-}+1\right] c(R_{I_--\frac{1}{2}})
+\sum_{i=I_-}^{I_0-1} \left[\vphantom{x_{I_-}} x_{i+1}-x_i\right] c(R_{i+\frac{1}{2}})
+\left[\vphantom{x_{I_-}} \xi^n-x_{I_0}\right]
\\
=&
\left[\vphantom{x_{I_-}} x_{I_0+1}-\xi^n\right]
+\sum_{i=I_0+1}^{I_+-1} \left[\vphantom{x_{I_-}} x_{i+1}-x_i\right] c(R_{i+\frac{1}{2}})
+\left[\vphantom{x_{I_-}} 1-x_{I_+}\right] c(R_{I_++\frac{1}{2}}).
\end{split}
\]
Take the time derivative of the above equation
\begin{align*}
&\dot{x}_{I_-} \, c(R_{I_--\frac{1}{2}})
+\left[x_{I_-}+1\right] c'(R_{I_--\frac{1}{2}}) \, \dot{R}_{I_--\frac{1}{2}}\\
&+\sum_{i=I_-}^{I_0-1} \left\{
\left[\vphantom{x_{I_-}} \dot{x}_{i+1}-\dot{x}_i\right] c(R_{i+\frac{1}{2}})
+\left[\vphantom{x_{I_-}} x_{i+1}-x_i\right] c'(R_{i+\frac{1}{2}}) \, \dot{R}_{i+\frac{1}{2}}
\right\}
+\left[\vphantom{x_{I_-}} \dot{\xi}^n-\dot{x}_{I_0}\right]
\\
=&
\left[\vphantom{x_{I_-}} \dot{x}_{I_0+1}-\dot{\xi}^n\right]
+\sum_{i=I_0+1}^{I_+-1} \left\{
\left[\vphantom{x_{I_-}} \dot{x}_{i+1}-\dot{x}_i\right] c(R_{i+\frac{1}{2}})
+\left[\vphantom{x_{I_-}} x_{i+1}-x_i\right] c'(R_{i+\frac{1}{2}}) \, \dot{R}_{i+\frac{1}{2}}
\right\}
\\
&-\dot{x}_{I_+} \, c(R_{I_++\frac{1}{2}})
+\left[\vphantom{x_{I_-}} 1-x_{I_+}\right] c'(R_{I_++\frac{1}{2}}) \, \dot{R}_{I_++\frac{1}{2}},
\end{align*}
namely
\begin{align*}
2 \, \dot{\xi}^n=&
-\dot{x}_{I_-} \, c(R_{I_--\frac{1}{2}})
-\left[x_{I_-}+1\right] c'(R_{I_--\frac{1}{2}}) \, \dot{R}_{I_--\frac{1}{2}}\\
&-\sum_{i=I_-}^{I_0-1} \left\{
\left[\vphantom{x_{I_-}} \dot{x}_{i+1}-\dot{x}_i\right] c(R_{i+\frac{1}{2}})
+\left[\vphantom{x_{I_-}} x_{i+1}-x_i\right] c'(R_{i+\frac{1}{2}}) \, \dot{R}_{i+\frac{1}{2}}
\right\}
+\dot{x}_{I_0}
\\
&+\dot{x}_{I_0+1}
+\sum_{i=I_0+1}^{I_+-1} \left\{
\left[\vphantom{x_{I_-}} \dot{x}_{i+1}-\dot{x}_i\right] c(R_{i+\frac{1}{2}})
+\left[\vphantom{x_{I_-}} x_{i+1}-x_i\right] c'(R_{i+\frac{1}{2}}) \, \dot{R}_{i+\frac{1}{2}}
\right\}
\\
&-\dot{x}_{I_+} \, c(R_{I_++\frac{1}{2}})
+\left[\vphantom{x_{I_-}} 1-x_{I_+}\right] c'(R_{I_++\frac{1}{2}}) \, \dot{R}_{I_++\frac{1}{2}}.
\end{align*}
Hence, from  \eqref{eq:DevinTownsend} and \eqref{eq:DevinTownsend2} we obtain
\begin{align*}
2 \, \dot{\xi}^n
=&-\dot{x}_{I_-} \, c(R_{I_--\frac{1}{2}})
+\left[x_{I_-}+1\right] c'(R_{I_--\frac{1}{2}}) \, R_{I_--\frac{1}{2}} \, \frac{\dot{x}_{I_-}-\dot{x}_{I_--1}}{x_{I_-}-x_{I_--1}}\\
&-\sum_{i=I_-}^{I_0-1} \left[\vphantom{x_{I_-}} \dot{x}_{i+1}-\dot{x}_i\right] \Upsilon(R_{i+\frac{1}{2}})
+\dot{x}_{I_0}
+\dot{x}_{I_0+1}
+\sum_{i=I_0+1}^{I_+-1} \left[\vphantom{x_{I_-}} \dot{x}_{i+1}-\dot{x}_i\right] \Upsilon(R_{i+\frac{1}{2}})
\\
&-\dot{x}_{I_+} \, c(R_{I_++\frac{1}{2}})
-\left[\vphantom{x_{I_-}} 1-x_{I_+}\right] c'(R_{I_++\frac{1}{2}}) \, R_{I_++\frac{1}{2}} \, \frac{\dot{x}_{I_++1}-\dot{x}_{I_+}}{x_{I_++1}-x_{I_+}}
\\
=&\left\{
\left[x_{I_-}+1\right]
\, \frac{\dot{x}_{I_-}-\dot{x}_{I_--1}}{x_{I_-}-x_{I_--1}}
- \dot{x}_{I_-}
\right\}
c'(R_{I_--\frac{1}{2}}) \, R_{I_--\frac{1}{2}}
\\
&-\sum_{i=I_-}^{I_0} \dot{x}_{i} \, \Upsilon(R_{i-\frac{1}{2}})
+\sum_{i=I_-}^{I_0-1} \dot{x}_{i} \, \Upsilon(R_{i+\frac{1}{2}})
+\dot{x}_{I_0}\\
&+\dot{x}_{I_0+1}
+\sum_{i=I_0+2}^{I_+} \dot{x}_{i} \, \Upsilon(R_{i-\frac{1}{2}})
-\sum_{i=I_0+1}^{I_+} \dot{x}_{i} \, \Upsilon(R_{i+\frac{1}{2}})\\
&-\left\{ \dot{x}_{I_+} +\left[\vphantom{x_{I_-}} 1-x_{I_+}\right] \, \frac{\dot{x}_{I_++1}-\dot{x}_{I_+}}{x_{I_++1}-x_{I_+}} \right\}
c'(R_{I_++\frac{1}{2}}) \, R_{I_++\frac{1}{2}}.
\end{align*}
Finally, by applying \eqref{eq:FTL} we conclude the proof.
\end{proof}

\subsection{Proof of Theorem~\ref{teo:1}}\label{sec:mainproof1}

Fix an initial datum $\bar\rho$ in $\mathcal{S}$.
In \cite{AmadoriGoatinRosini} it is proven that the function $\rho \in \mathcal{S}$ defined for $x \in [0,1)$ as the entropy solution to the initial-boundary value problem
\[
\begin{cases}
\rho_t + f(\rho)_x =0 & \text{if } t>0,~ x \in (0,1),\\
\rho(t,0) = \rho(t,1) = 0 & \text{if }  t >0,\\
\rho(0,x) = \bar\rho(x) & \text{if }  x \in (0,1),
\end{cases}
\]
is an entropy solution of \eqref{eq:model} in the sense of Definition~\ref{def:entropy_solution}, with $\xi \equiv 0$.
Moreover, as already pointed out in the introduction, the entropy solution to the above initial-boundary problem coincides with the entropy solution to the Cauchy problem
\[
\begin{cases}
\rho_t + f(\rho)_x =0 & \text{if }  t>0,~ x \in \R,\\
\rho(0,x) = \bar\rho_e(x) & \text{if } x \in \R,
\end{cases}
\]
where $\bar{\rho}_e$ is obtained from $\bar{\rho}$ by extending it to zero outside $(-1,1)$.
Now, adopting the atomization procedure described in subsection \ref{sec:particle}, it is clear that the discrete turning point will be located at zero for all times due to the symmetry of the particle system. Therefore, the particles will split into two sets which will not change in time, with the two particles nearest the turning point getting further and further away from each other. Hence, by the results in \cite{DF_rosini}, we obtain that the FTL-Hughes model \eqref{eq:discrete_Hughes} converges to the entropy solution of the above Cauchy problem as $m$ goes to zero and $N$ to infinity, and this concludes the proof.

\subsection{Proof of Theorem~\ref{teo:2}}\label{sec:mainproof2}

Observe that the function $\rho \mapsto \, \Upsilon(\rho)$ defined in \eqref{eq:upsilon} is strictly decreasing, indeed
\begin{align*}
&\, \Upsilon'(\rho) = \left[c\left(\rho\right) - c'\left(\rho\right) \rho\right]'
=
- c''\left(\rho\right) \rho < 0
&\text{for all } \rho \in \left(0,\rho_{\max}\right],
\end{align*}
because $c$ is assumed to be convex, and $\Upsilon(0)=1$.
Hence ${\tv}_c(\rho) \leq L \, \tv(\rho)$, where
\begin{align*}
&{\tv}_c(\rho) \doteq \tv(\Upsilon(\rho)) = \tv\left(c\left(\rho\right) - c'\left(\rho\right) \rho \right)\,,
\end{align*}
and $L$ is defined in \eqref{eq:lipschitz}.

\begin{proposition}\label{prop:1}
Let $L$ be as defined in \eqref{eq:lipschitz} and $C$ be as defined in \eqref{eq:Cgrande}. If
\begin{equation}\label{eq:condition1}
  \frac{v_{\max}}{2}
\left[\vphantom{\sum} L\, {\tv}(\bar{\rho})
+
3 \, C\right]
<
v(\rho_{\max}),
\end{equation}
then no particles change direction and the problem \eqref{eq:FTL} admits a global-in-time solution.
\end{proposition}

\begin{proof}
Directly from Lemma~\ref{lem:dotxi} we deduce that
\[
|\dot{\xi}^n| \leq Q[\bar\rho]\,\]
with
\begin{equation}\label{eq:cone}
  Q[\bar\rho]\doteq
\frac{v_{\max}}{2}
\left[\vphantom{\int} L \,{\tv}(\bar{\rho})
+
3 \, C\right].
\end{equation}
Indeed, it is easy to obtain from the estimate proved in Lemma~\ref{lem:dotxi} that
\begin{align*}
2 \, \dot{\xi}^n
&\le
v_{\max} \left[
\tv_c(R)
+2 \, c'(R_{I_--\frac{1}{2}}) \, R_{I_--\frac{1}{2}}
+1-\Upsilon(R_{I_0+\frac{3}{2}})\right]
\\
&\le
v_{\max} \left[\vphantom{\sum} \tv_c(R)+3 \, C\right]
\end{align*}
and
\begin{align*}
2 \, \dot{\xi}^n
&\ge
-v_{\max} \left[
\tv_c(R)
+2 \, c'(R_{I_++\frac{1}{2}}) \, R_{I_++\frac{1}{2}}
+1-\Upsilon(R_{I_0-\frac{1}{2}})\right]
\\
&\ge
-v_{\max} \left[\vphantom{\sum} \tv_c(R)+3 \, C\right].
\end{align*}
Moreover, by the discrete maximum principle proven in~\cite[Lemma~1]{DF_rosini}, as long as no collisions occur between a particle and the turning point, we have that
\begin{align*}
&\dot{x}_{I_0} = -v\left(\dfrac{m}{x_{I_0}-x_{I_0-1}}\right)
\le -v\left(\max_{i\in\{1,\ldots,I_0\}}\left\{\frac{m}{\bar{x}_{i} - \bar{x}_{i-1}}\right\}\right)
\le -v(\rho_{\max}),\\
&\dot{x}_{I_0+1} = v\left(\dfrac{m}{x_{I_0+2}-x_{I_0+1}}\right)
\ge v\left(\max_{i\in\{I_0,\ldots,N-1\}}\left\{\frac{m}{\bar{x}_{i+1} - \bar{x}_{i}}\right\}\right)
\ge  v(\rho_{\max}).
\end{align*}
In conclusion we proved that the particles do not belong to the cone where $\xi^n$ evolves.
As a consequence, no particle crosses the turning curve, namely no particle changes direction.
Hence problem \eqref{eq:FTL} can be split in two follow-the-leader problems, for which global existence follows from~\cite[Lemma~1]{DF_rosini}. Moreover, the $\BV$ contraction estimate proven in \cite[Proposition 5]{DF_rosini} implies that the $\tv_c(R)\leq L\,\tv(\bar{\rho})$, and this concludes the proof.
\end{proof}

\begin{remark}
We remark that in the special case $c(\rho)=1/v(\rho)$ and $v(\rho)=1-\rho$, condition \eqref{eq:condition1} is satisfied only if $\rho_{\max}<\check{\rho}\sim 0.265$. This condition is analogous to the one obtained in \cite{AmadoriGoatinRosini}.
\end{remark}

In the $n\rightarrow +\infty$ limit it is clear that $\bar\xi^n$ converges to the unique $\bar{\xi}\in (-1,1)$ such that
\[\int_{-1}^{\bar\xi}c(\bar\rho(y)) \, {\d}y =\int_{\bar\xi}^1 c(\bar\rho(y)) \, {\d}y \,.\]
A crucial consequence of the result in Proposition \ref{prop:1} is that condition \eqref{eq:condition1} is independent of $n$. Let $Q[\bar\rho]$ be the constant defined in \eqref{eq:cone}, and consider the cone $\mathcal{C}\doteq \{(t,x) \in \R_+ \times \R \colon |x-\bar\xi| < Q[\bar\rho] \, t\}$. First of all we get $|\dot{\xi}^n|\leq Q[\bar\rho]$, which shows that the curve $\xi^n$ converges up to a subsequence strongly in $\C0([0,T];\R)$ for all $T\geq 0$ to some $\xi\in \C0([0,T];\R)$, and the limit turning curve $\mathcal{T}=\left \{(t,x) \in \R_+ \times \R \colon x=\xi(t)\right \}$  is entirely contained in $\mathcal{C}$. Since no particles are placed in $\mathcal{C}$, the discrete density $R(t,x)$ converges to zero strongly in $\Lloc1(\mathcal{C})$. By taking $\delta>0$ arbitrarily small, we can apply the same procedure of~\cite[Theorem~3]{DF_rosini} to prove that the discretized density $R(t,x)$ defined in \eqref{eq:density} converges strongly in $\Lloc1$ towards a function $\rho_R\in \L\infty([\delta,+\infty)\times \R)$ satisfying Kruzhkov's entropy condition \cite{Kruzhkov} for the conservation law $\rho_t+[\rho \, v(\rho)]_x = 0$ on $([\delta,+\infty)\times [0,+\infty))$. Similarly, $R(t,x)$ converges strongly in $\Lloc1$ towards a function $\rho_L\in \L\infty([\delta,+\infty)\times \R)$ satisfying Kruzhkov's entropy condition \cite{Kruzhkov} for the conservation law $\rho_t-[\rho \, v(\rho)]_x = 0$ on $([\delta,+\infty)\times (-\infty,0])$.
By setting
\[\rho(t,x)=
\begin{cases} \rho_L(t,x) & \text{if } x<\xi(t), \\
\rho_R(t,x) & \text{if } x>\xi(t),
\end{cases}\]
we can integrate \eqref{eq:equilibriocosto} in time on $[0,+\infty)$, multiply by an arbitrary test function $\varphi\in \Cc0([0,+\infty))$, and get in the $n\rightarrow +\infty$ limit (up to a subsequence)
\[\int_{0}^{+\infty}\left[\int_{-1}^{\xi(t)}c(\rho(t,y))\, {\d}y -\int_{\xi(t)}^1 c(\rho(t,y))\, {\d}y\right]\varphi(t) \, {\d} t = 0\,,\]
which shows that \eqref{eq:integral_equation} is satisfied for almost every $t\geq 0$.

In order to prove that $\rho$ is the unique entropy solution to \eqref{eq:model} according to Definition~\ref{def:entropy_solution}, we only need to prove that the initial condition $\bar\rho$ is achieved as $t\searrow 0$ at least in a weak measure sense. Now, it is immediately seen that $R(0,x)$ only differs from the initial atomization $\bar{R}(x)$ on the interval $[\bar{x}_{I_0},\bar{x}_{I_0+1})$. Moreover,
recalling the definition of $\bar{R}$ in \eqref{eq:R_bar} and of $R$ in (\ref{eq:density}), we have
\[\int_{\bar{x}_{I_0}}^{\bar{x}_{I_0+1}}|\bar{R}(y)-R(0,y)| \, \d y = \int_{\bar{x}_{I_0}}^{\bar{x}_{I_0+1}} \bar{R}(y) \,\d y\]
and the last term above clearly tends to zero as $n\rightarrow +\infty$. Since $\bar{R}(x)$ converges to $\bar\rho$ in the sense of measures as $n\rightarrow +\infty$, we have that also $R(0,x)$ approaches $\bar\rho$ as $n\rightarrow +\infty$. This concludes the proof of Theorem~\ref{teo:2}.

\subsection{Particle approximations of Riemann data}\label{sec:riemann}

In this subsection we consider the particle approximation of an initial condition
\[\bar\rho(x)=\begin{cases}
\rho_- & \text{if } x\in[-1,0], \\
\rho_+ & \text{if } x\in (0,1],
\end{cases}\]
with $\rho_-<\rho_+$. In order to simplify the computations, we first fix the number of particles $[-1,0)$ to be equal to $N^-$, and the number of particles on $(0,1]$ equal to $N^+$, with $N^+>N^-$. We call $N \doteq N^+ + N^-$. We then fix the particle mass $m>0$ such that $m<1/N^+$, and obtain the two Riemann values
\begin{align*}
&\rho_-= m \, N^-,&
&\rho_+= m \, N^+.
\end{align*}
We set the initial $N+1$ particle positions as follows
\begin{align*}
\bar{x}_0\doteq -1\,,&&
&\bar{x}_{i+1}\doteq\bar{x}_i + \frac{1}{N^-}\,,&
&i\in\{1,\ldots,N^- -2\}\,,\\
 \bar{x}_N\doteq 1\,,&&
&\bar{x}_i \doteq \bar{x}_{i+1}-\frac{1}{N^+}\,,&
&i \in \{ N^- ,\ldots, N-1\}\,.
\end{align*}
The initial position of the discrete turning point is computed as
\[\bar\xi =\frac{c(\rho_+)-c(\rho_-)}{2 \, c(\rho_+)}\in (0,1/2)\,.\]
%In order to have exactly $N$ particles, and therefore have a discrete density with total mass equal to $\rho_- + \rho_+$, we remove the particle closest to $\bar\xi$ (which possibly occupies the position of $\bar\xi$). If $\bar\xi$ is set in the middle point between two particles, we remove the one set on the right hand side. Clearly, $\bar\xi$ will never occupy the position of any particle initially.
%We set $I$ as the index of the nearest particle to $\bar\xi$ on the left hand side.
In order to avoid the situation in which $\bar\xi$ coincides with the initial position of a particle, we remove that particle from the system in such a case. Lemma \ref{lem:dotxi} implies
\begin{align*}
\dot{\xi}(0) &=\frac{1}{2}\left\{ v_{\max}\left[c'(\rho_-)\rho_- - c'(\rho_+)\rho_+\right]+v(\rho_-)\left[\Upsilon(\rho_-)-\Upsilon(\rho_+)\right]\right\}\\
  & = \frac{1}{2}\left\{\left[v_{\max}-v(\rho_-)\right]\left[\Upsilon(\rho_+)-\Upsilon(\rho_-)\right] +v_{\max}\left[c(\rho_-)-c(\rho_+)\right]\right\}.
\end{align*}
The assumptions on $c$ and $v$ clearly imply that $\dot{\xi}(0)<0$. Therefore, the avoidance of a collision between the turning point $\xi(t)$ and a neighbor particle is guaranteed under the assumption
\begin{equation}\label{eq:riemann1}
  F(\rho_-,\rho_+)\doteq v_{\max}\left[c'(\rho_-)\rho_- - c'(\rho_+)\rho_+\right]+v(\rho_-)\left[\Upsilon(\rho_-)-\Upsilon(\rho_+)\right]+2 \, v(\rho_+)>0\,.
\end{equation}
It is immediately seen that condition \eqref{eq:riemann1} is satisfied when $\rho_+=\rho_-$ and $\rho_+<1$. On the other hand, if $\rho_-=0$ \eqref{eq:riemann1} may not hold.

In the special case
\begin{align*}
& c(\rho)=\frac{1}{v(\rho)},&
&v(\rho)=1-\rho,
\end{align*}
treated in \cite{DiFrancescoMarkowichPietschmannWolfram,AmadoriDiFrancesco}, the function $F$ in \eqref{eq:riemann1} reads
\[F(\rho_-,\rho_+)=\frac{(2-\rho^+)(1-2\rho^+)}{1-\rho^+}+\frac{(\rho^-)^2}{(1-\rho^-)^2}+\frac{\rho^-(1-2\rho^+)}{(1-\rho^+)^2}\,.\]
Let us test condition \eqref{eq:riemann1} on $\rho_-=0$. We get
\[0<F(0,\rho_+) =\frac{(2-\rho^+)(1-2\rho^+)}{1-\rho^+}\,,\]
and the above is satisfied on $\rho_+\in [0,1]$ only if
\[\rho_+<\frac{1}{2}\,.\]
We now study the monotonicity of $F$ with respect to $\rho_-$. We get
\[\frac{\partial F}{\partial \rho_-} (\rho_-,\rho_+)= \Upsilon(\rho_+)-\Upsilon(\rho_-) +c'(\rho_-) +\rho_-^2 \, c''(\rho_-)\,.\]
Hence, $\frac{\partial F}{\partial \rho_-} (\rho_-,\rho_+)>0$ if and only if
\[\Upsilon(\rho_+)>G(\rho_-)\]
with
\[G(\rho_-) \doteq c(\rho_-)-\rho_- \, c'(\rho_-) -c'(\rho_-) -\rho_-^2 \, c''(\rho_-) = -\frac{2 \, \rho_-}{(1-\rho_-)^3}\,.\]
Now, a quick inspection shows that $\Upsilon(\rho)>G(\rho)$ for all $\rho\in [0,1)$. Moreover, since $\Upsilon(\rho)=\frac{1-2\rho}{(1-\rho)^2}$ in this case, we infer the existence of a curve $\rho_+=\varphi(\rho_-)$ with $\varphi(0)=1/2$, $\varphi(1)=1$, and $\varphi'(\rho_-)>0$ on $\rho_-\in [0,1)$, such that \eqref{eq:riemann1} holds if and only if $\rho_-\leq \rho_+\leq \varphi(\rho_-)$. Such a condition is automatically satisfied if $\rho_+$ and $\rho_-$ are both less than $1/2$.

\section{Numerical simulations}\label{sec:numerics}

In this section we present some numerical simulations of the particle method described above, and compare the tests with classical methods available in the literature. More precisely, we solve the particle system \eqref{eq:FTL} using the Runge-Kutta MATLAB solver ODE23 with initial mesh sizes automatically determined by the total number of particles $N$ and the initial density values. We then compare the particle scheme simulations with solutions of \eqref{eq:hughes_continuum} obtained via Godunov scheme.

An important remark has to be stated about the boundary conditions. As described in section \ref{sec:particle}, we do not impose any boundary condition on the particle method. The two leading particles $x_0$ and $x_{N}$ move with maximal velocity towards the opposite directions. For the Godunov method the boundary conditions are assigned as follows. We create two extra ghost cells, one just at the left of $-1$ and one juts at the right of $1$. In all our numerical computation we set the value of $\rho$ in those cells to zero, to mimic `perfect exits'. In practice, any value that results in characteristic speeds pointing out of the domain will provide the solution, since the upwind scheme will not make use of such values (except to check that the characteristic actually point out of the domain). This procedure is used for convenience, since it simplifies the implementation of boundary conditions a lot.

Particular attention is devoted to the turning point evolution in the particle simulations, obtained by discretizing \eqref{eq:equilibriocosto}. Since no boundary conditions are imposed for the particle method, particles are free to exit the domain following the evolution of the two leaders, whereas only the particles still inside the domain bias the evolution of the turning point.

We test the atomization algorithm described in section~\ref{sec:particle} considering some  examples introduced in the literature for various initial data $\bar{\rho}$, see \cite{DiFrancescoMarkowichPietschmannWolfram} and \cite{El-KhatibGoatinRosini}.
In each of the example reported, the choice for the cost function is
\begin{align*}
& c(\rho)=\frac{1}{v(\rho)},&
&v(\rho)=1-\rho,
\end{align*}
and we show time evolution of the discrete density $R(t,x)$ \eqref{eq:density} in the domain $\Omega=\left(-1,1\right)$. In \figurename~\ref{fig:Test 1} and \figurename~\ref{fig:Test 2} we consider two constant initial conditions $\bar{\rho}(x)=0.25$ and $\bar{\rho}(x)=0.6$ respectively. In the latter case, as we expect, the simulation shows the formation of two rarefaction waves centered at $x=-1$ and $x=1$, with trace $R(t,\pm 1)\sim 1/2$. In the former case the rarefaction waves exit the boundary, hence we see three constant states inside $\Omega$.

\begin{figure}[htbp]
\begin{minipage}[l]{6cm}
\begin{center}
\includegraphics[width=6cm]{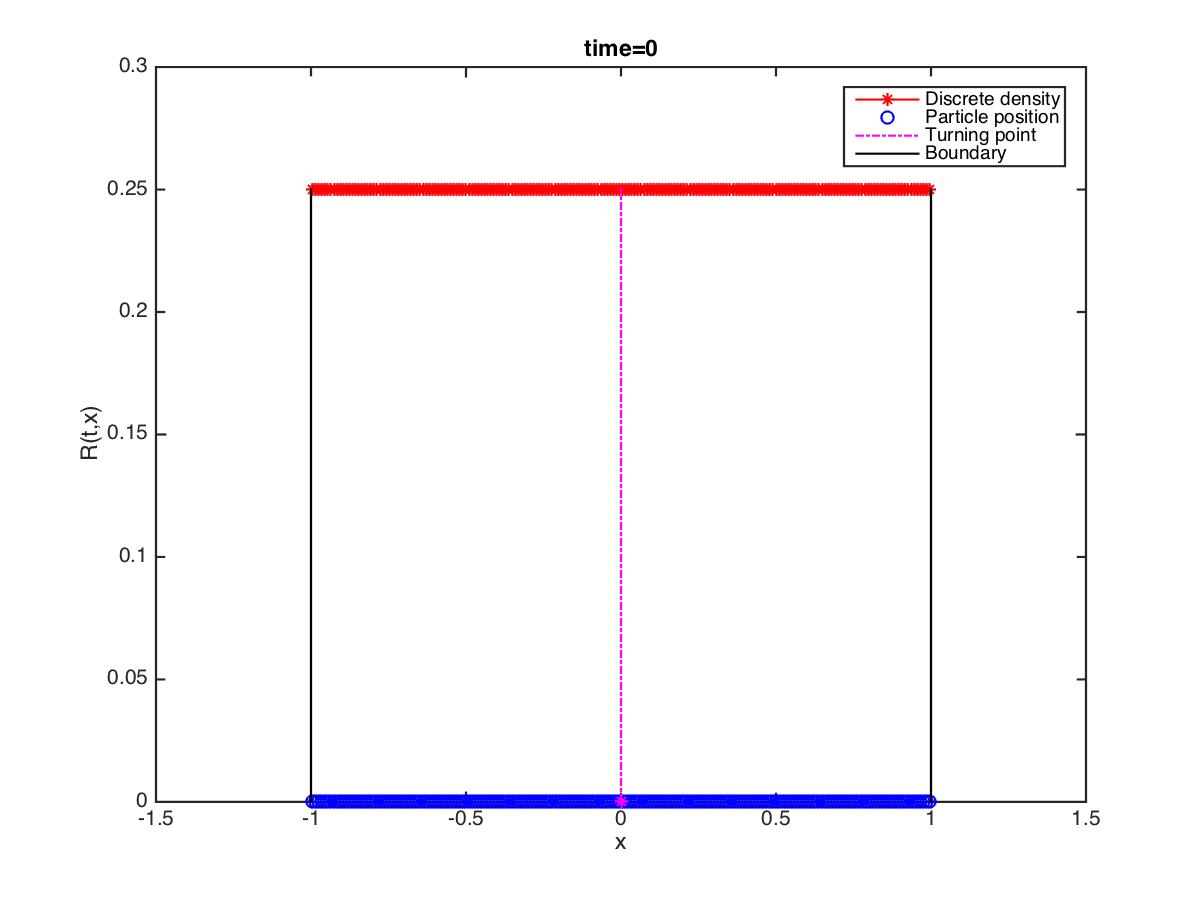}
\end{center}
\end{minipage}%
\begin{minipage}[c]{6cm}
\begin{center}
\includegraphics[width=6cm]{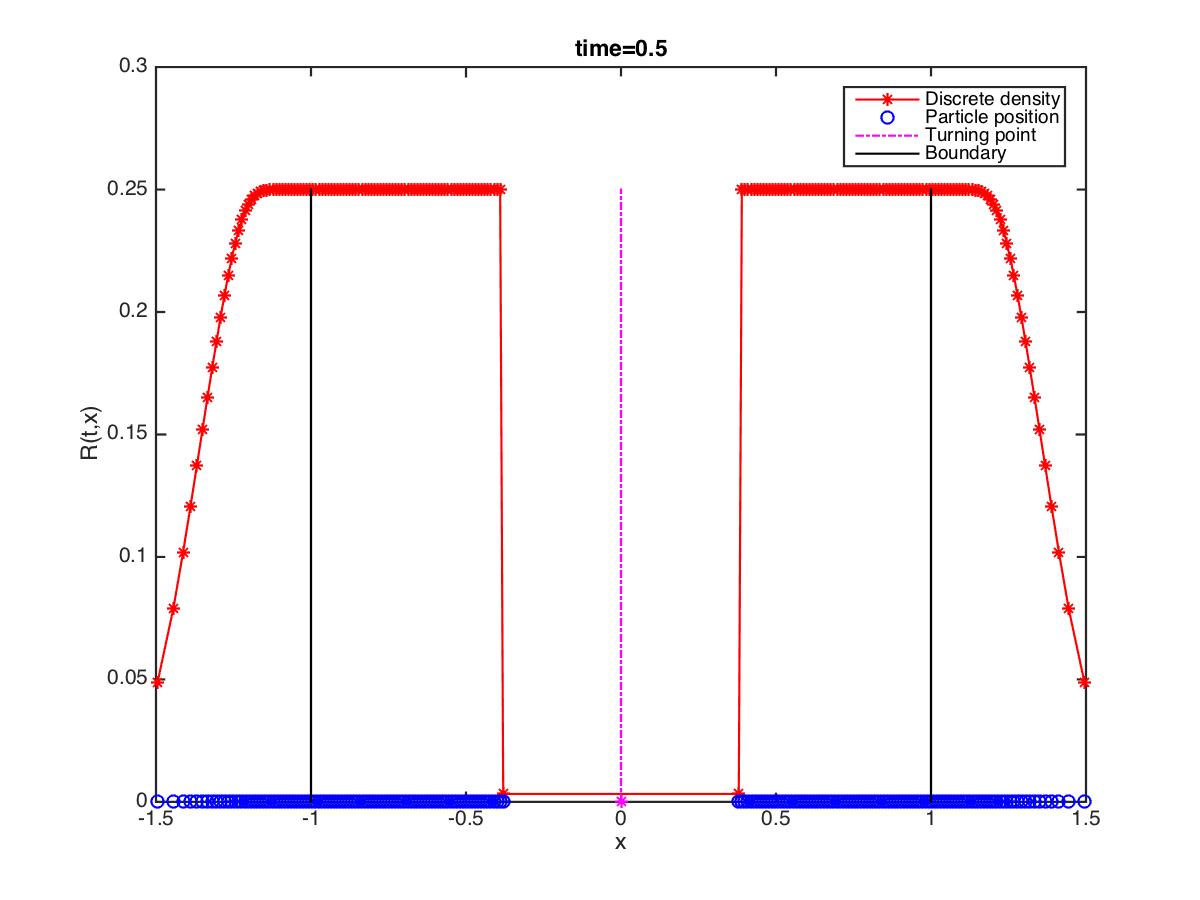}
\end{center}
\end{minipage}
\begin{minipage}[c]{6cm}
\begin{center}
\includegraphics[width=6cm]{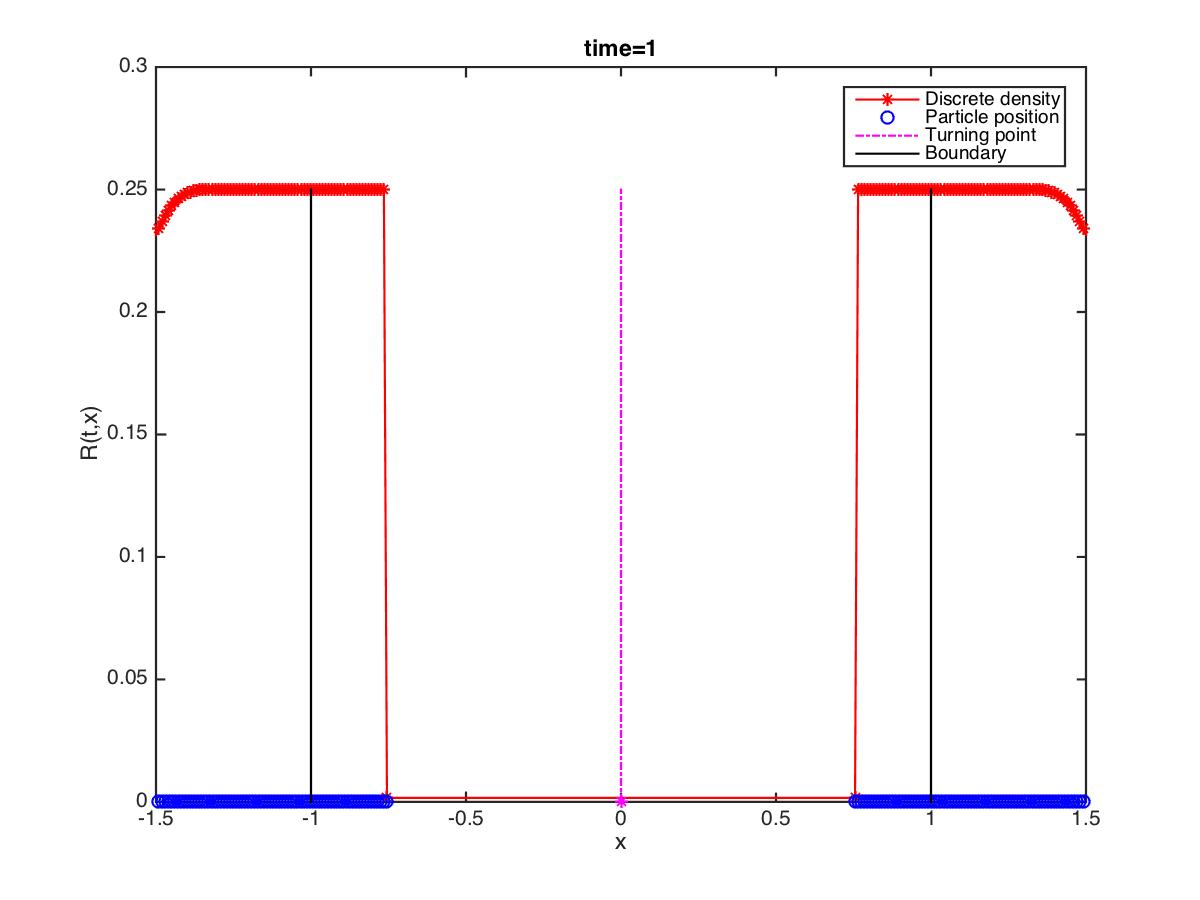}
\end{center}
\end{minipage}
\caption{Evolution of $R(t,x)$ with initial data $\bar{\rho}(x)=0.25$ at times $t=0$, $t=0.5$ and $t=1$. In the particle simulations the blu dots represent particles positions, whereas the red line is the discretized density. The magenta vertical line describes the turning point evolution. \label{fig:Test 1}}
\end{figure}
\begin{figure}[htbp]
\begin{minipage}[l]{6cm}
\begin{center}
\includegraphics[width=6cm]{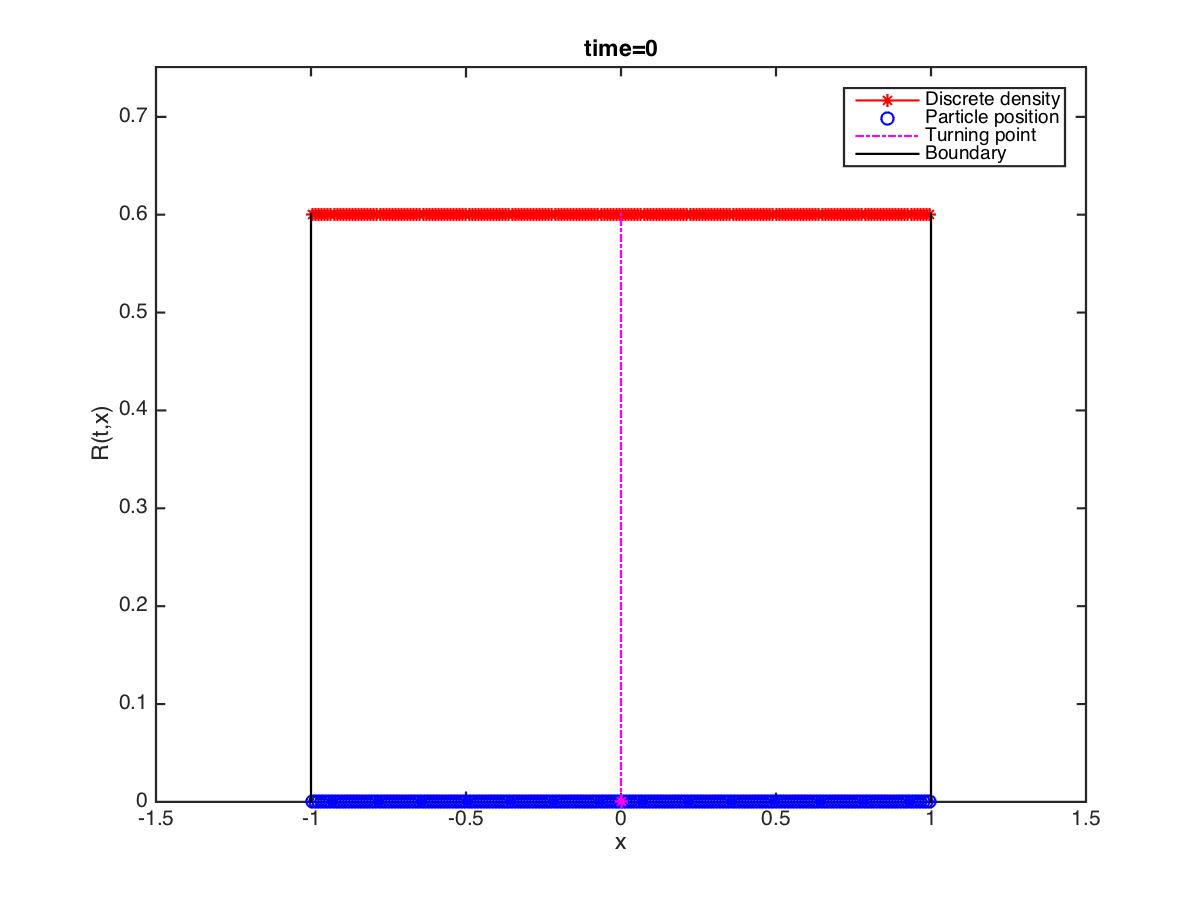}
\end{center}
\end{minipage}%
\begin{minipage}[c]{6cm}
\begin{center}
\includegraphics[width=6cm]{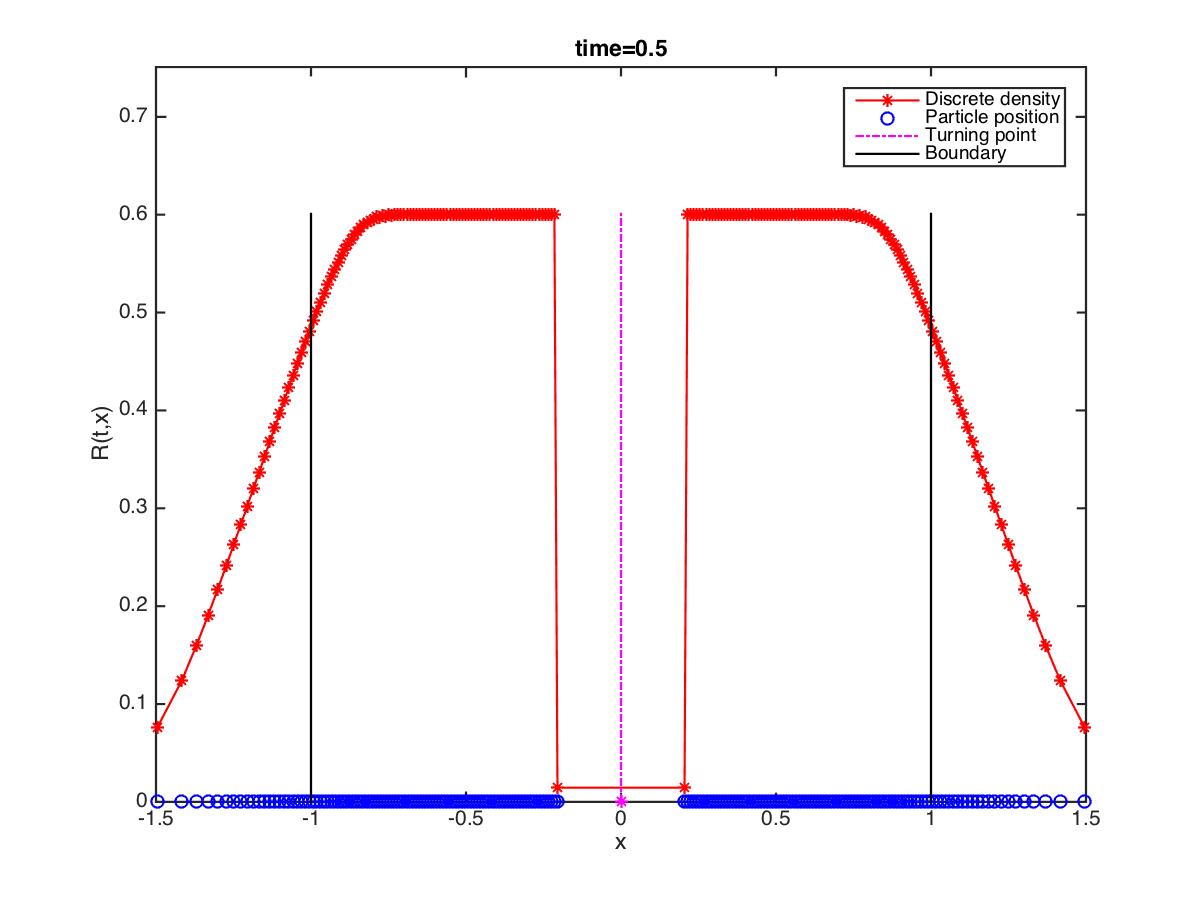}
\end{center}
\end{minipage}
\begin{minipage}[c]{6cm}
\begin{center}
\includegraphics[width=6cm]{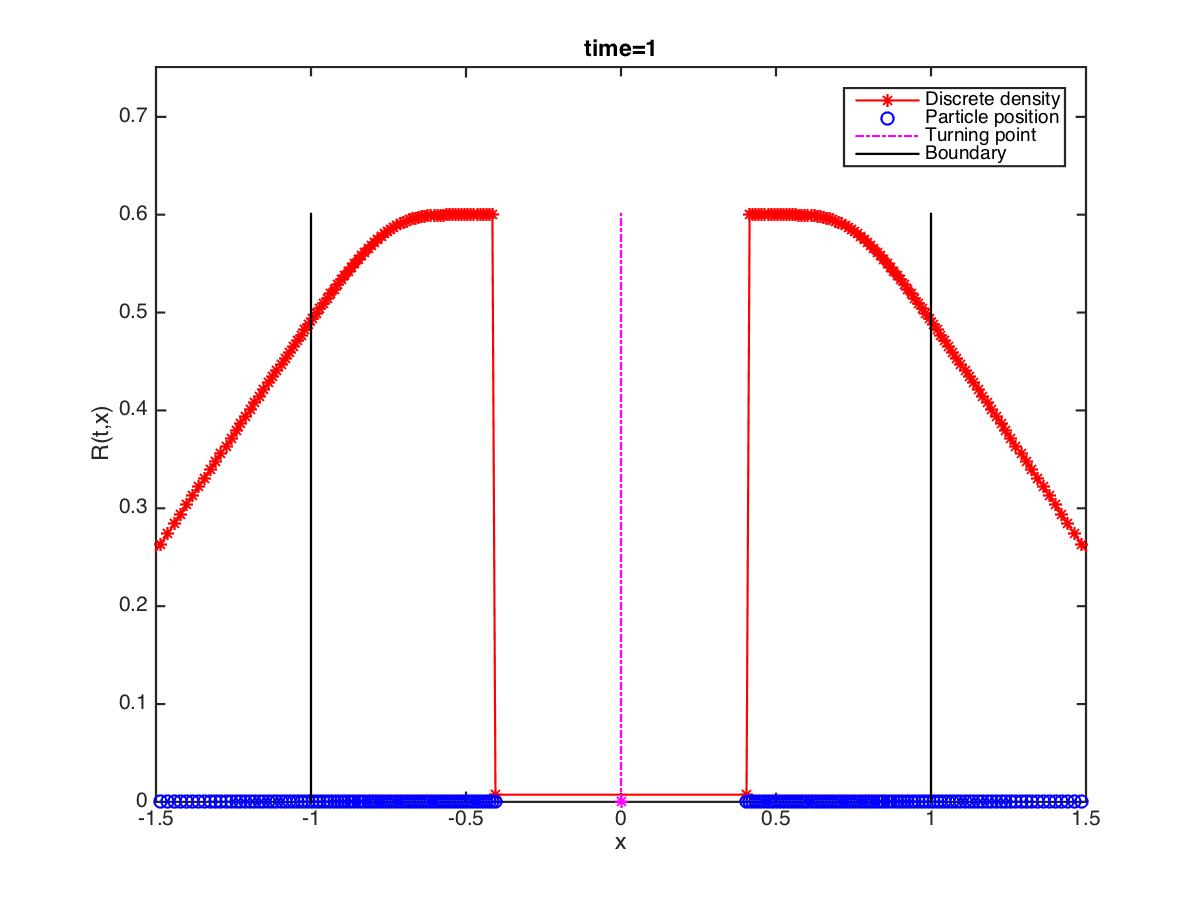}
\end{center}
\end{minipage}
\caption{Evolution of $R(t,x)$ with initial data $\bar{\rho}(x)=0.6$ at times $t=0$, $t=0.5$ and $t=1$. \label{fig:Test 2}}
\end{figure}
In \figurename~\ref{fig:Test 3} we consider the Riemann problem with initial condition
\begin{equation}\label{Rp}
\bar{\rho}(x)=\begin{cases}
 0.45 &\text{if } x\leq 0, \\
 0.55 &\text{if } x>0,
\end{cases}
\end{equation}
for which the condition \eqref{eq:riemann1} in section \ref{sec:riemann} is satisfied. In this case, a vacuum region is formed around the turning point. In \figurename~\ref{fig:Test 4} we consider the initial datum
\begin{equation}\label{Rp2}
\bar{\rho}(x)=\begin{cases}
 0.1 &\text{if } x\leq 0, \\
 0.9 &\text{if } x>0\,,
\end{cases}
\end{equation}
for which condition \eqref{eq:riemann1} is \emph{not} satisfied. In this case, the turning point collides with its left neighbour particle.
In order to compare our method with the tests performed in \cite{DiFrancescoMarkowichPietschmannWolfram,GoatinMimault}, in \figurename~\ref{fig:Test 5} we consider the three-step initial condition
\begin{equation}\label{Steps}
\bar{\rho}(x)=\begin{cases}
 0.8 &\text{if } -0.8\leq x\leq -0.5, \\
 0.6 &\text{if } -0.3\leq x\leq 0.3, \\
 0.9 &\text{if } 0.4\leq x\leq 0.75, \\
 0 & \text{otherwise}.
\end{cases}
\end{equation}
As shown in \figurename~\ref{fig:Test 4} and \figurename~\ref{fig:Masstr}, examples \eqref{Rp2} and \eqref{Steps}, exhibit the typical \emph{mass transfer} phenomenon occurring when the turning point is not surrounded by a vacuum region. In such a case, particles are crossing the turning point $\xi(t)$, and a \emph{non-classical shock} is formed around $\xi(t)$, see \cite[Remark 5]{AmadoriDiFrancesco}. These examples show that the particle method introduced in section \ref{sec:particle} has good chances to rigorously approximate the continuum solution to Hughes' model under less restrictive assumptions than the ones considered in theorems \ref{teo:1} and \ref{teo:2}. Indeed, a key requirement in order to have the particle scheme in section \ref{sec:particle} well posed and consistent in the limit is the \emph{avoidance of wild oscillations of the particle trajectories around the turning point}. Such a behaviour is ensured in the examples \eqref{Rp2} and \eqref{Steps}, and we believe that it can be obtained in a less restrictive setting for the initial conditions than the one we imposed in Theorems \ref{teo:1} and \ref{teo:2}. In support of this conjecture we performed several simulations with various initial conditions, and there is no evidence of such strange behavior. Of course a more detailed analysis is needed to rigorously prove convergence of the scheme.

In all the examples we set $N=200$ particles and plot the particle positions and the discrete densities. We highlight the interaction with the boundary and the evolution of the turning point.

\begin{figure}[htbp]
\begin{minipage}[l]{6cm}
\begin{center}
\includegraphics[width=6cm]{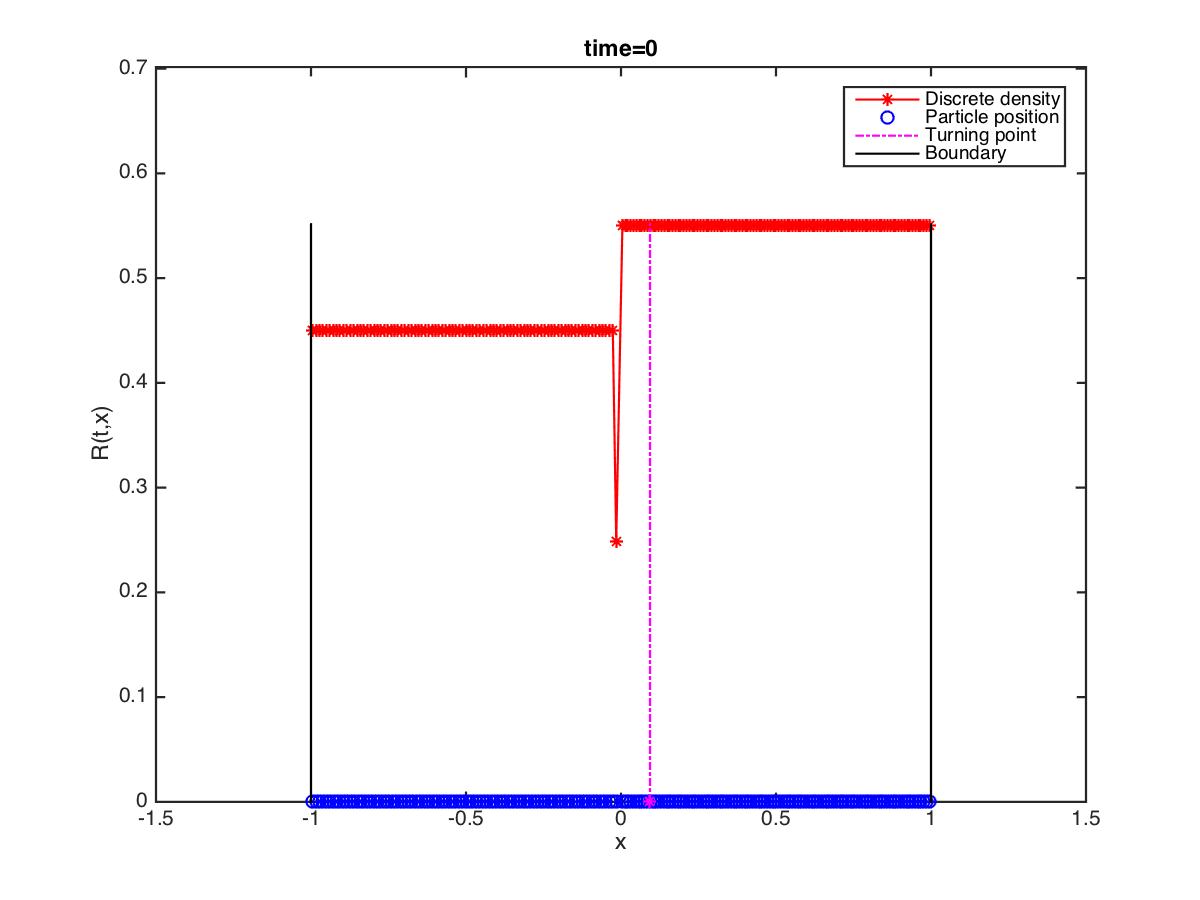}
\end{center}
\end{minipage}%
\begin{minipage}[c]{6cm}
\begin{center}
\includegraphics[width=6cm]{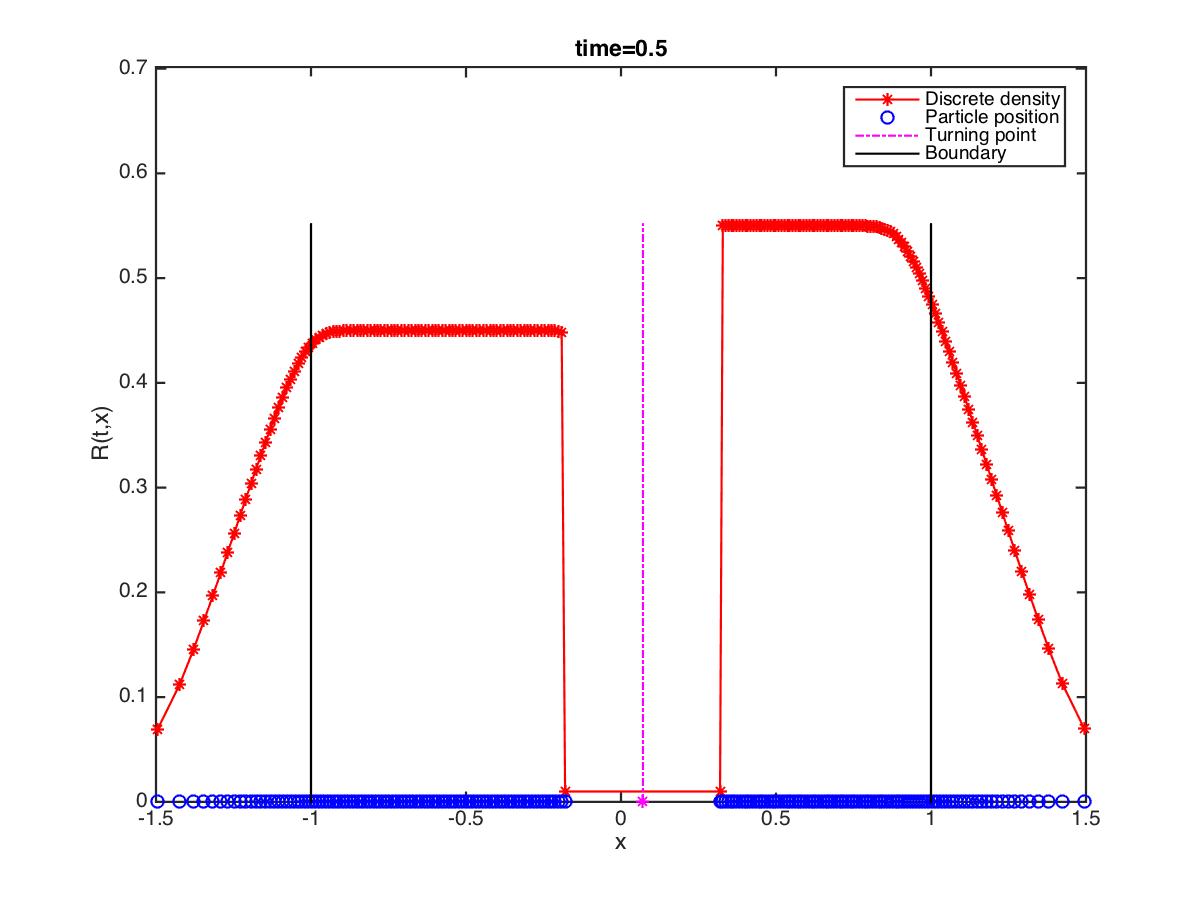}
\end{center}
\end{minipage}\\
\begin{minipage}[c]{6cm}
\begin{center}
\includegraphics[width=6cm]{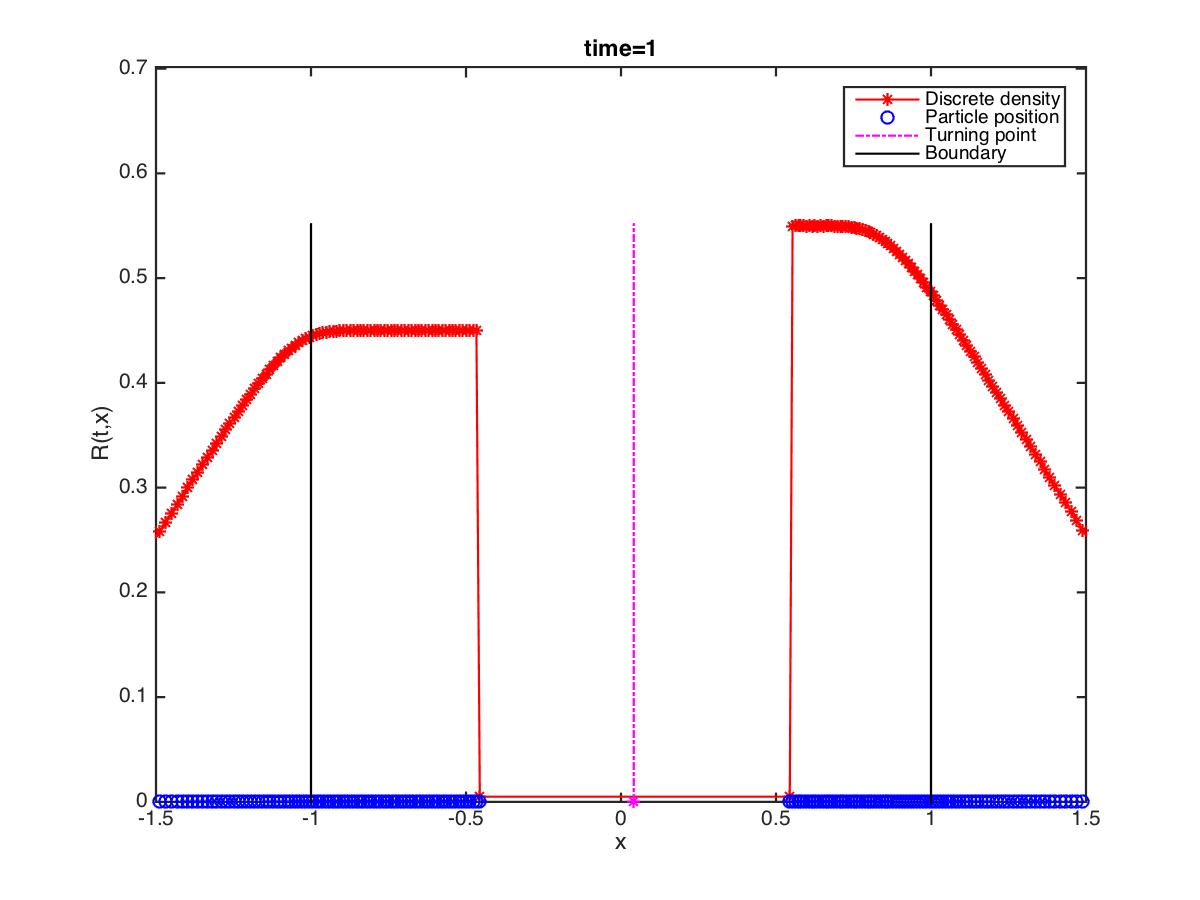}
\end{center}
\end{minipage}
\caption{Evolution of $R(t,x)$ with initial data $\bar{\rho}(x)$ given in \eqref{Rp}. \label{fig:Test 3}}
\end{figure}

\begin{figure}[htbp]
\begin{minipage}[l]{6cm}
\begin{center}
\includegraphics[width=6cm]{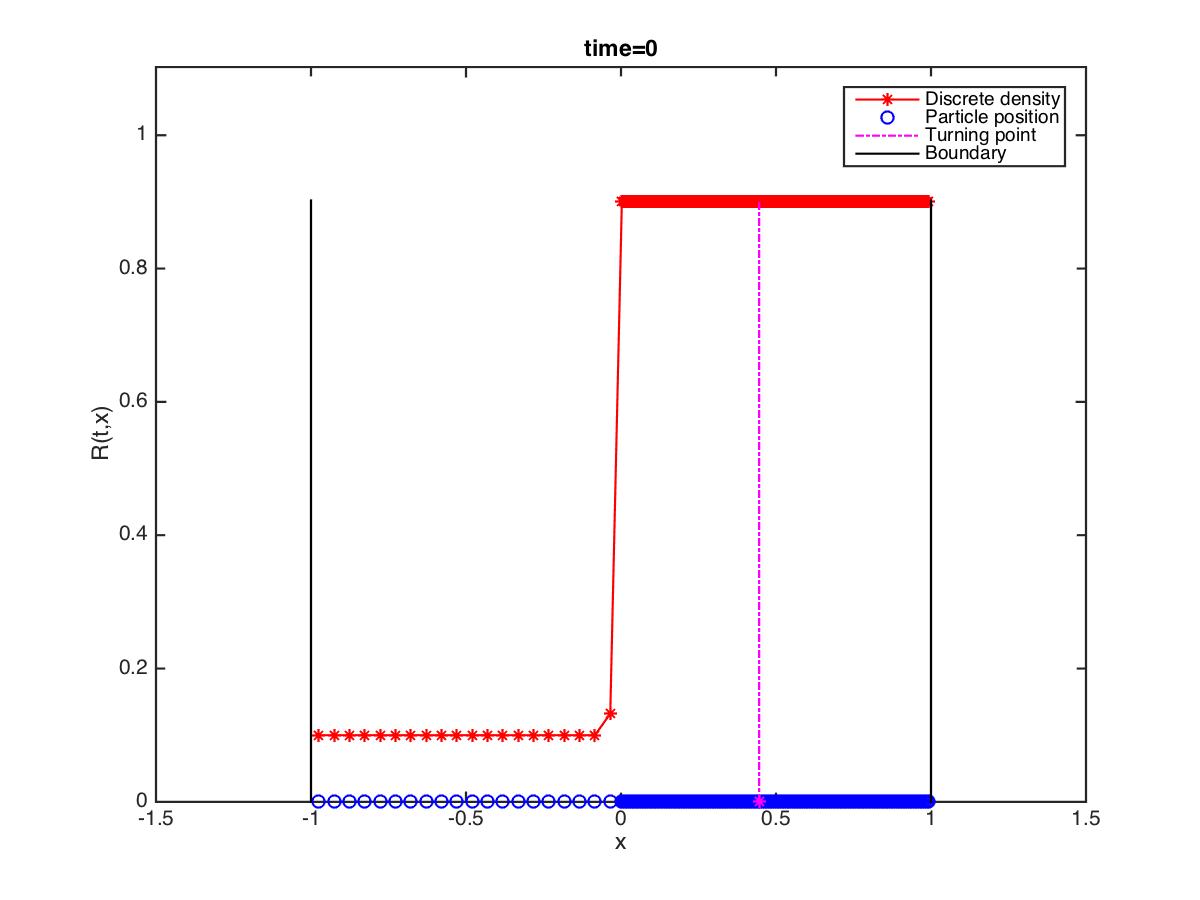}
\end{center}
\end{minipage}%
\begin{minipage}[c]{6cm}
\begin{center}
\includegraphics[width=6cm]{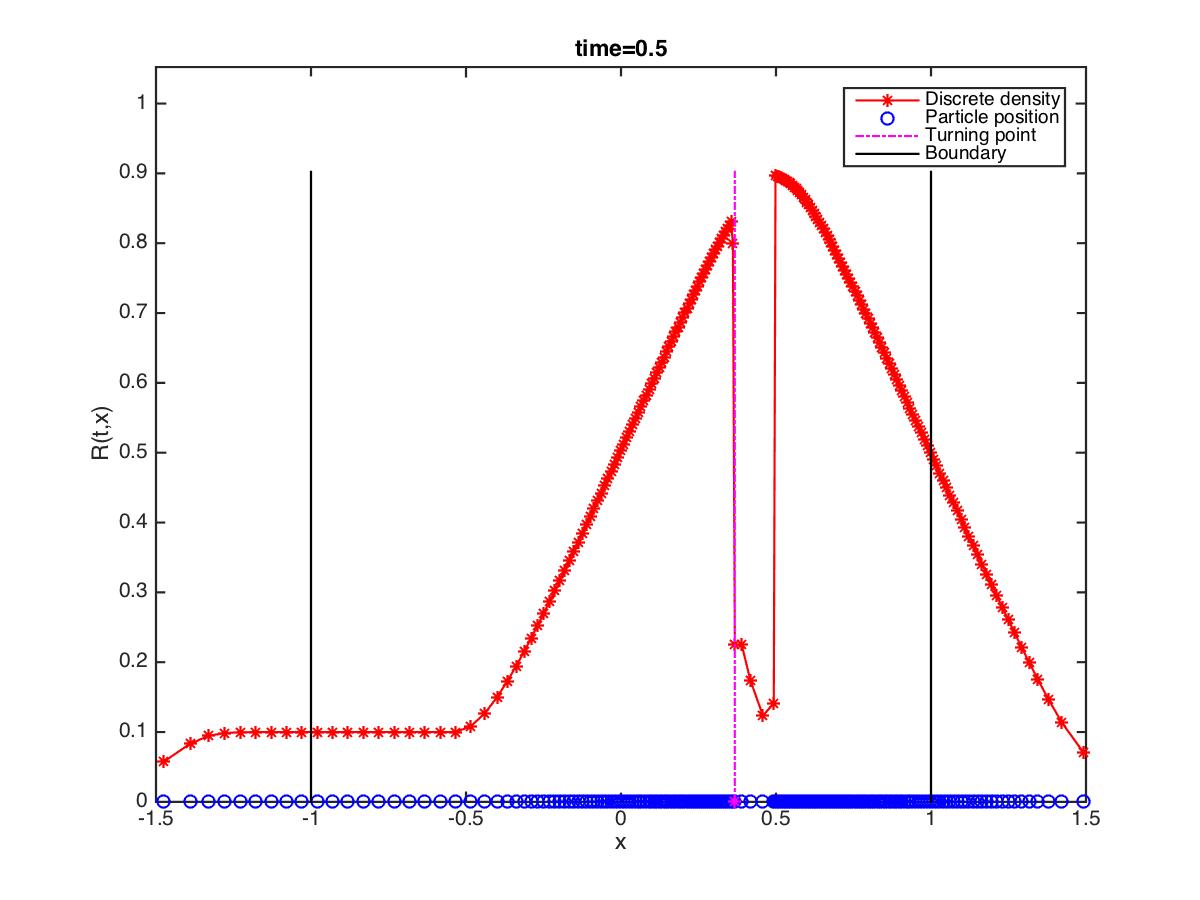}
\end{center}
\end{minipage}\\
\begin{minipage}[c]{6cm}
\begin{center}
\includegraphics[width=6cm]{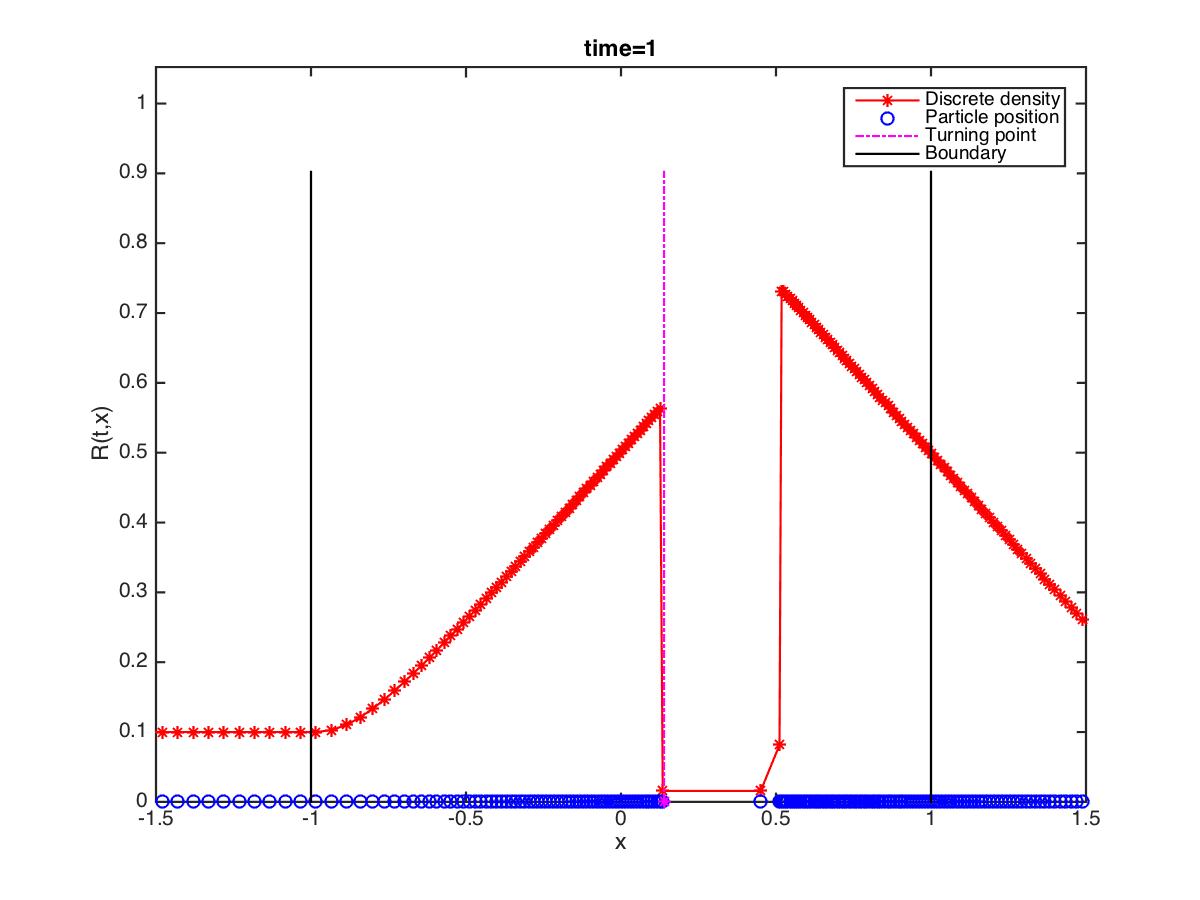}
\end{center}
\end{minipage}
\caption{Evolution of $R(t,x)$ with initial data $\bar{\rho}(x)$ given in \eqref{Rp2}. \label{fig:Test 4}}
\end{figure}

\begin{figure}[htbp]
\begin{minipage}[l]{6cm}
\begin{center}
\includegraphics[width=6cm]{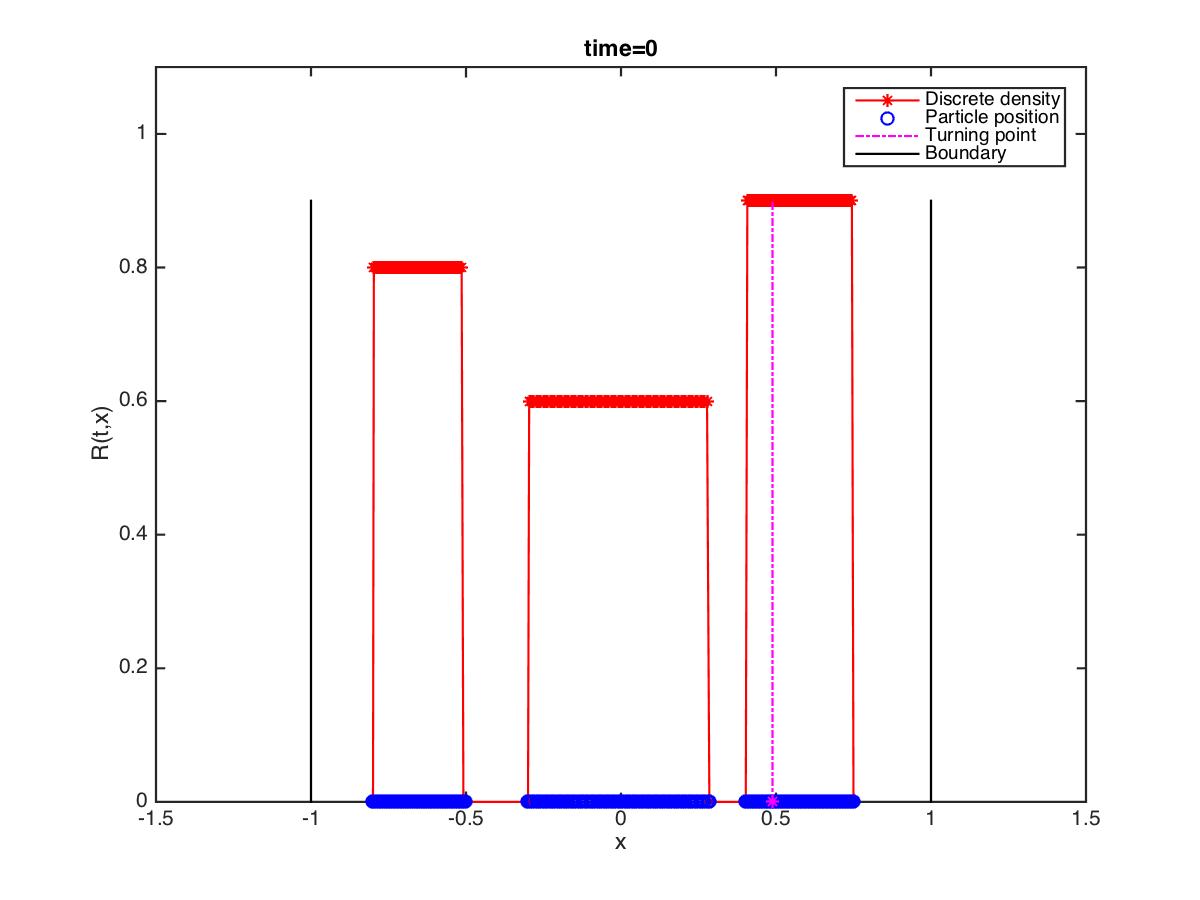}
\end{center}
\end{minipage}%
\begin{minipage}[c]{6cm}
\begin{center}
\includegraphics[width=6cm]{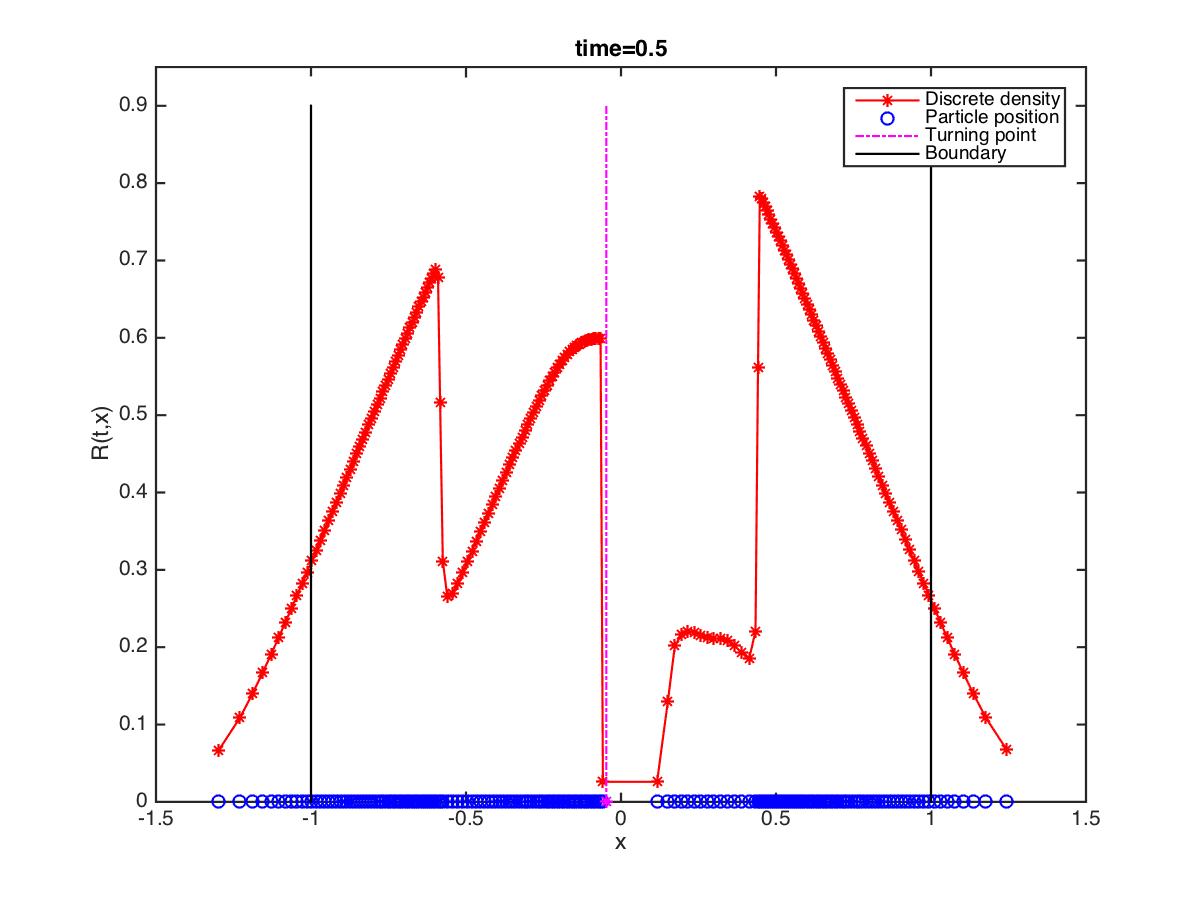}
\end{center}
\end{minipage}
\begin{minipage}[c]{6cm}
\begin{center}
\includegraphics[width=6cm]{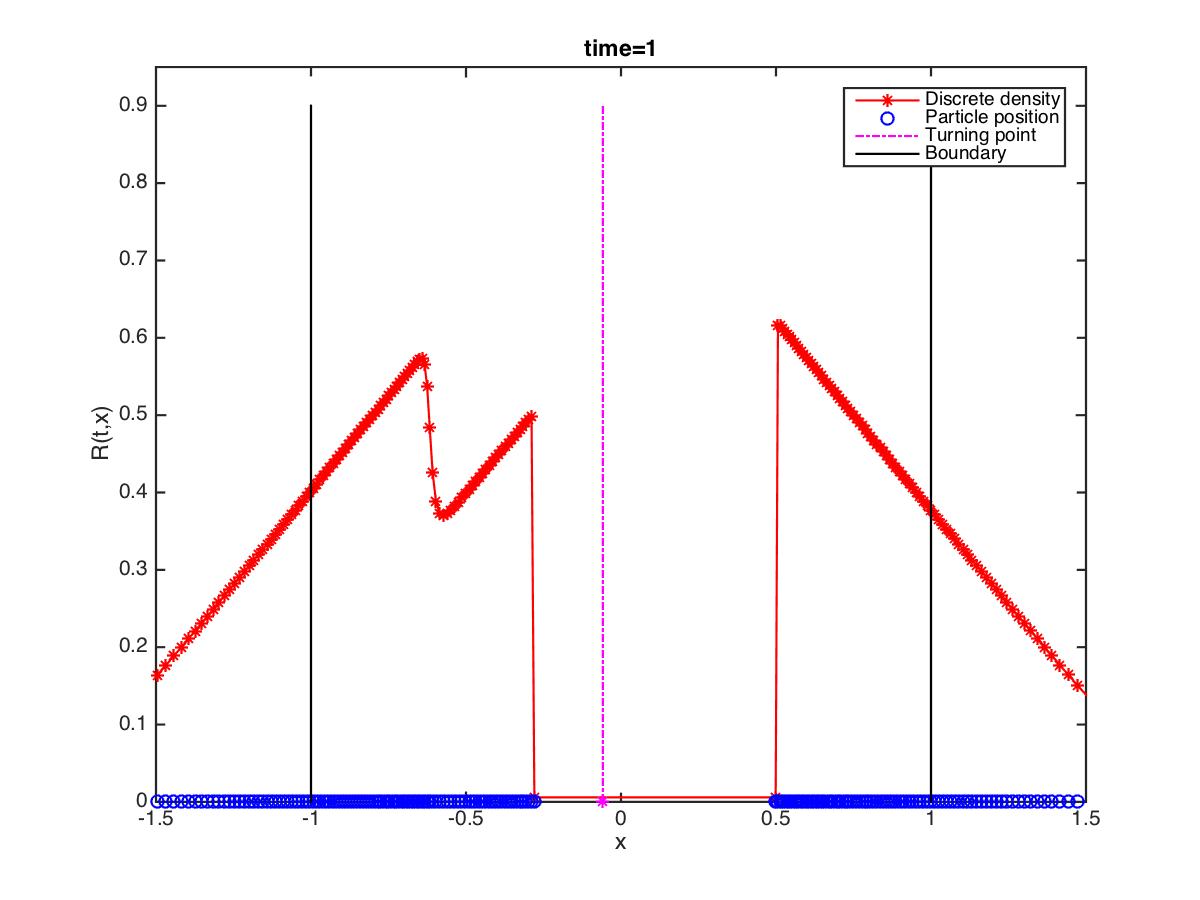}
\end{center}
\end{minipage}
\caption{Evolution of $R(t,x)$ with initial data $\bar{\rho}(x)$ given in \eqref{Steps}. \label{fig:Test 5}}
\end{figure}

\begin{figure}[htbp]
\begin{minipage}[l]{6cm}
\begin{center}
\includegraphics[width=6cm]{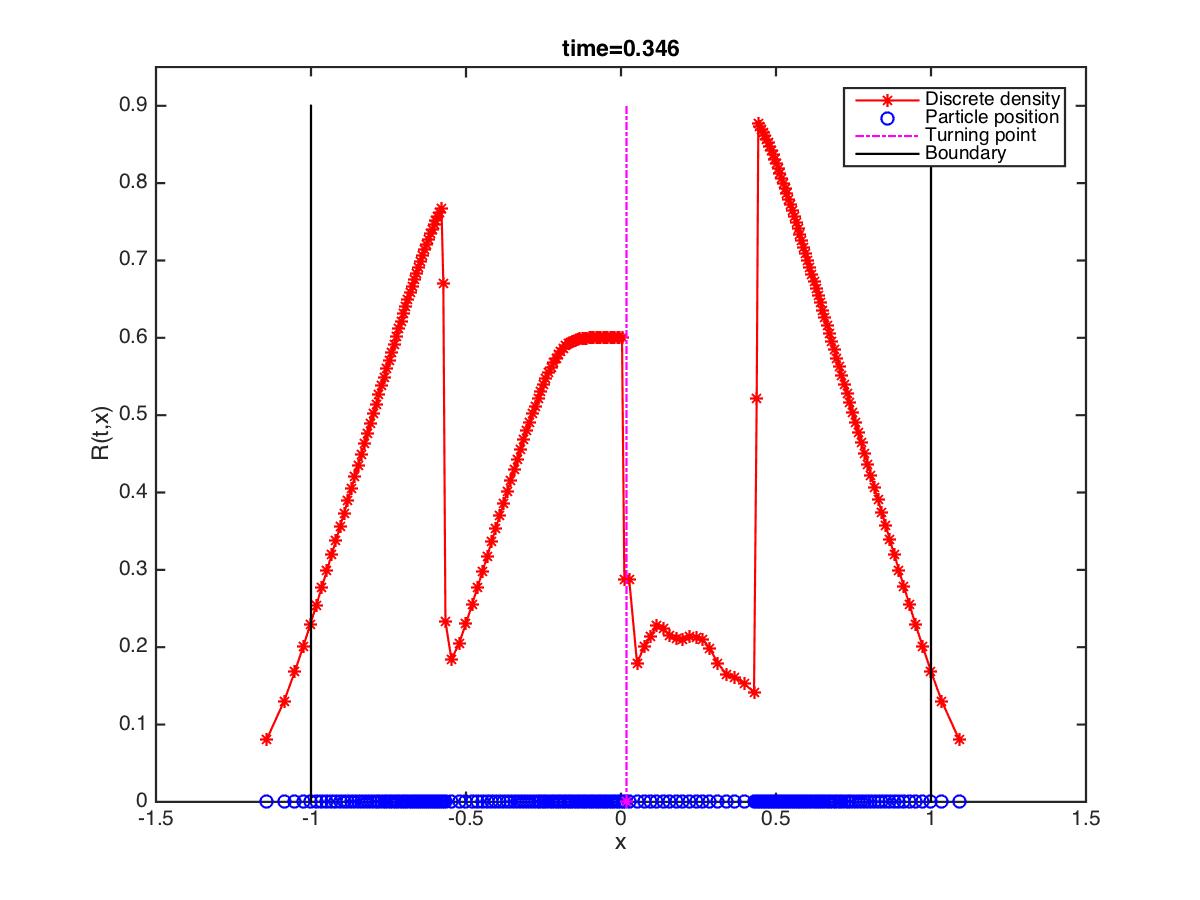}
\end{center}
\end{minipage}%
\begin{minipage}[c]{6cm}
\begin{center}
\includegraphics[width=6cm]{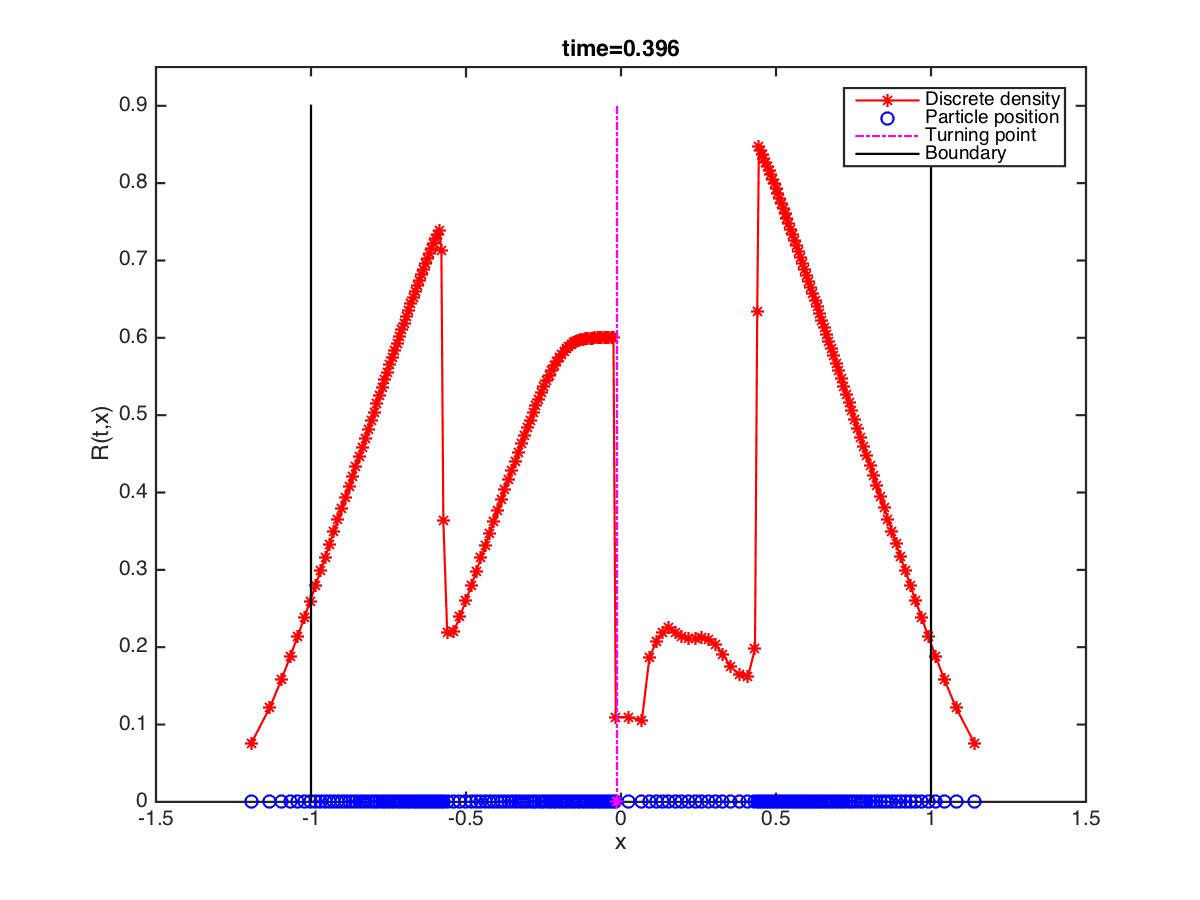}
\end{center}
\end{minipage}
%\begin{minipage}[c]{6cm}
%\begin{center}
%\includegraphics[width=6cm]{hughes3p_251}
%\end{center}
%\end{minipage}
\caption{Mass transfer across the turning point and \emph{non-classical shock} with initial data $\bar{\rho}$ given in \eqref{Steps}. \label{fig:Masstr}}
\end{figure}

In the literature, the model \eqref{eq:hughes_continuum} is usually solved in an iterative manner at each time step:
\begin{enumerate}
\item Given $\rho$, solve the eikonal equation.
\item Given the solution to the eikonal equation, solve the non-linear scalar conservation law using the Godunov, Lax-Friedrichs or ENO schemes (see \cite{DiFrancescoMarkowichPietschmannWolfram} and \cite{El-KhatibGoatinRosini}).
\end{enumerate}
A comparison betweeen the first-order particle method and a classical Godunov scheme is showed in \figurename~\ref{fig:Comp}, where we plot the solution of the previous simulations (at time $t=1$) both with the FTL-Hughes particle method and with the Godunov scheme. We set the spatial discretization according to the number of particles $N=200$ and the time discretization with $\Delta t=5\times10^{-2}$. Note that the time step is the same for both methods and is selected so as to respect the CFL condition for the Godunov method. Empirically, the observed time step restriction for the FTL-Hughes method \eqref{eq:discrete_Hughes} is much less severe than for the Godunov method applied to the Eulerian descriprion of the flow. However a detailed stability analysis of the method is beyond the scope of the present paper, and will be subject of future investigation.

It is evident from \figurename~\ref{fig:Comp2} that the two methods, though conceptually different, produce solutions which differ by a very small error. A more refined analysis of the error will be produced in a future work.

\begin{figure}[htbp]
\begin{minipage}[l]{6cm}
\begin{center}
\includegraphics[width=6cm]{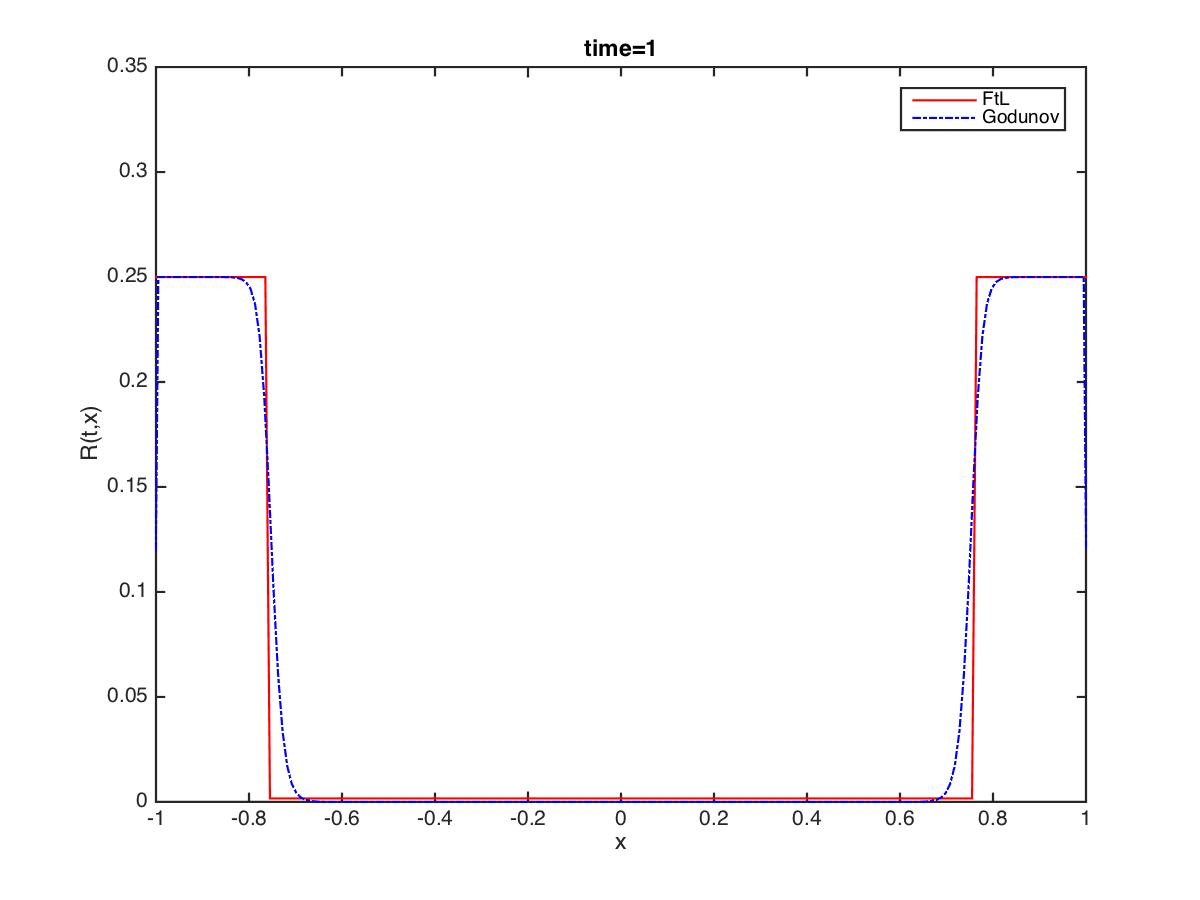}
\end{center}
\end{minipage}%
\begin{minipage}[c]{6cm}
\begin{center}
\includegraphics[width=6cm]{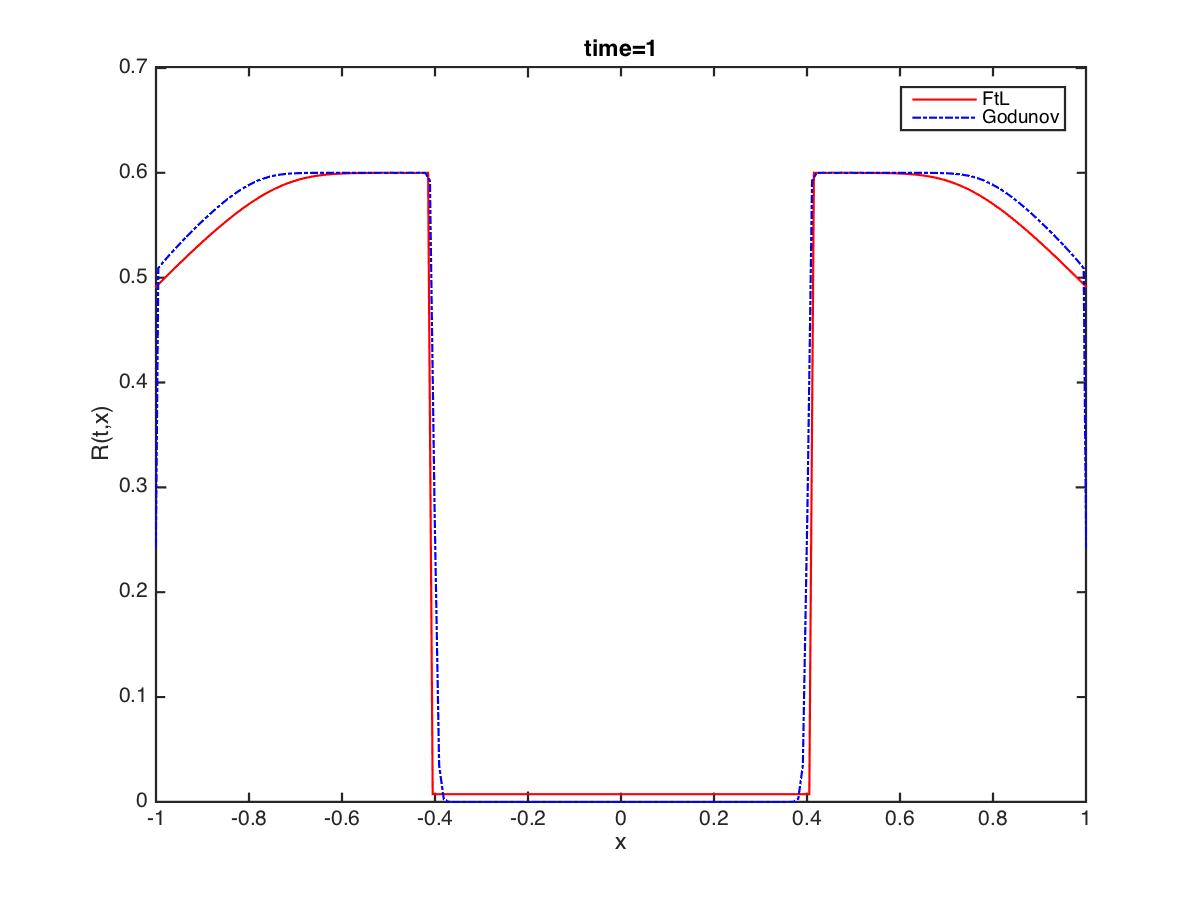}
\end{center}
\end{minipage}\\
\begin{minipage}[c]{6cm}
\begin{center}
\includegraphics[width=6cm]{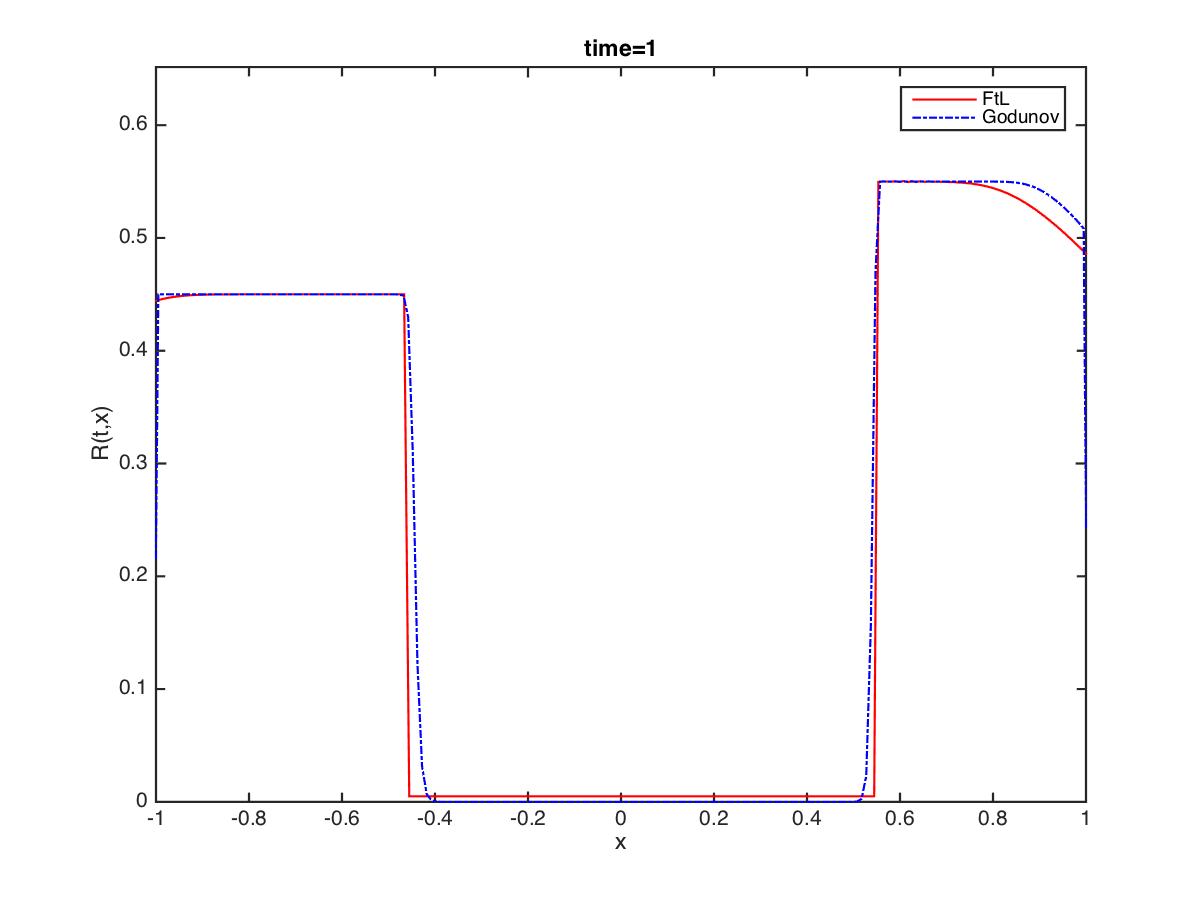}
\end{center}
\end{minipage}
\begin{minipage}[c]{6cm}
\begin{center}
\includegraphics[width=6cm]{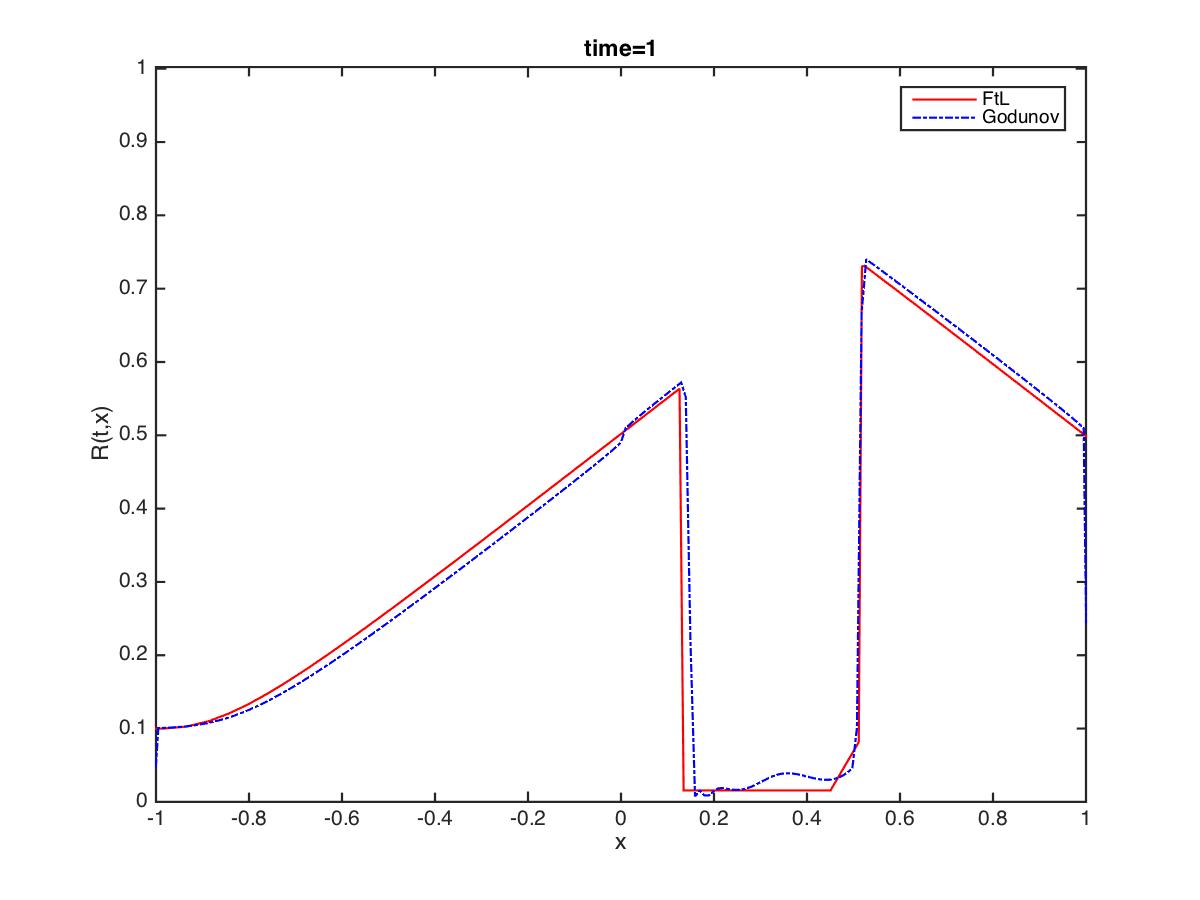}
\end{center}
\end{minipage}\\
\begin{minipage}[c]{6cm}
\begin{center}
\includegraphics[width=6cm]{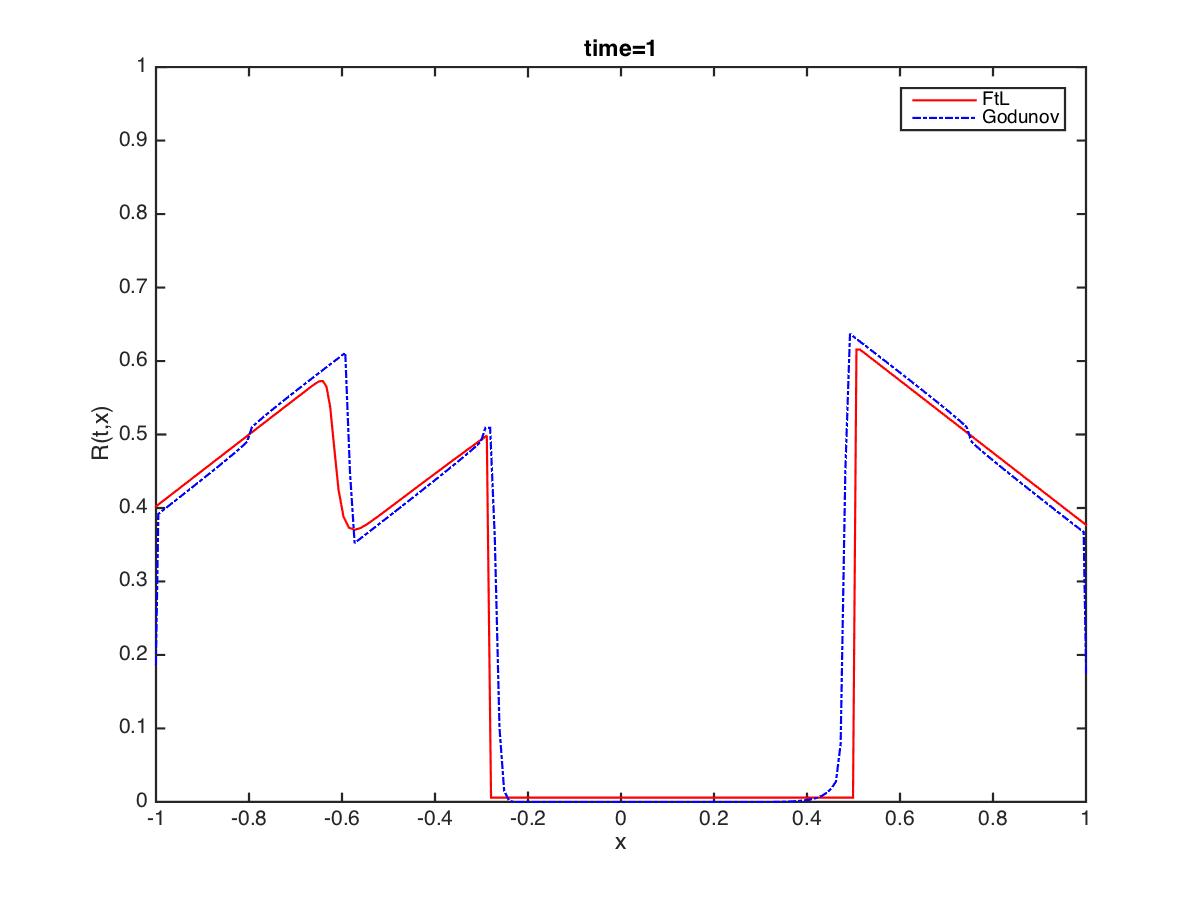}
\end{center}
\end{minipage}
\caption{Comparison between the Follow-the-Leader scheme (in red) and the Godunov scheme (in blu). \label{fig:Comp}}
\end{figure}

\begin{figure}[htbp]
\begin{center}
\includegraphics[width=8cm]{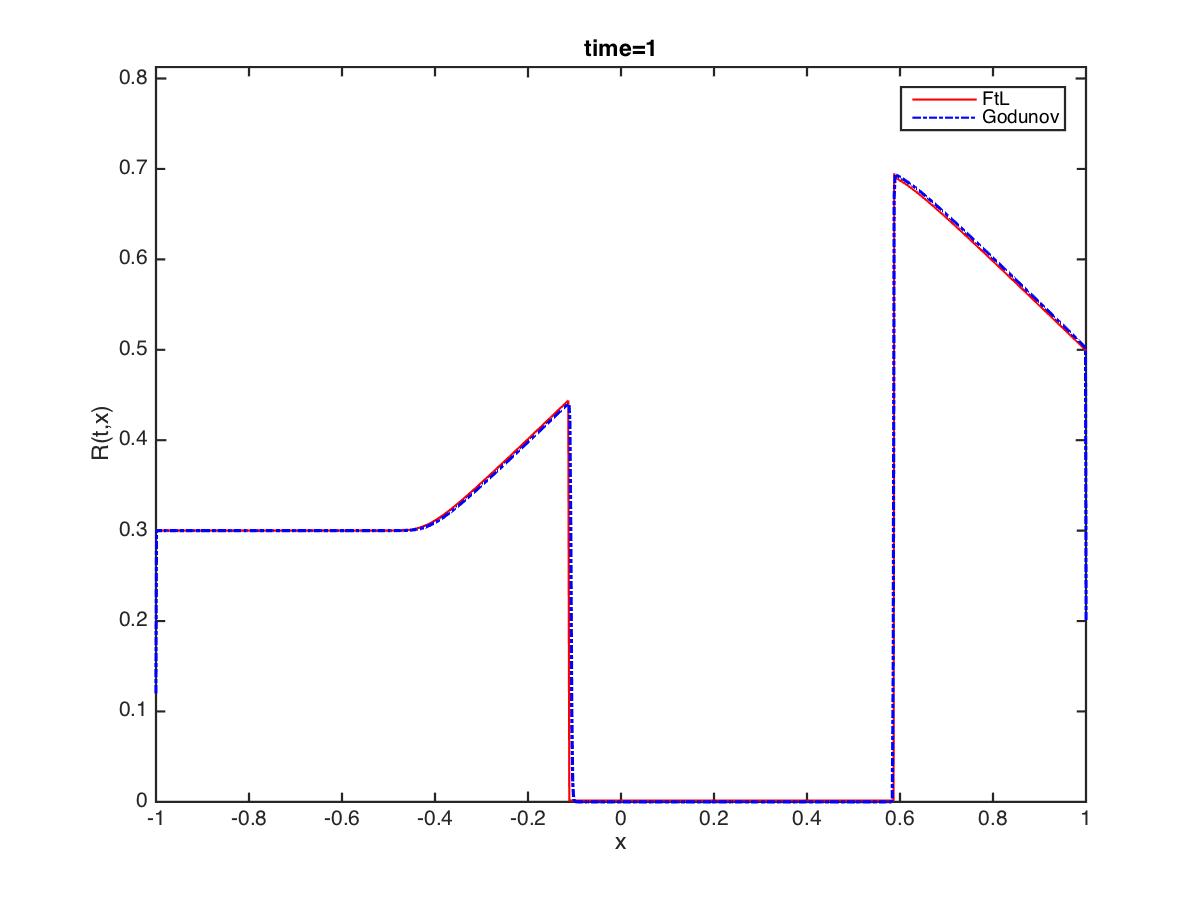}
\end{center}
\caption{Increasing the number of particles, and then the cells integration for the Godunov method, the agreement between the two methods greatly improves. Here we consider the initial datum $\bar{\rho}=0.3\times\,\mathbf{1}_{[-1,0]}+0.7\times \,\mathbf{1}_{(0,1]}$ and we set $N=1000$ with $1500$ time iterations. \label{fig:Comp2}}
\end{figure}

\section*{Conclusions}

The FTL-Hughes particle scheme introduced in section \ref{sec:particle} has been proven to rigorously converge to weak entropy solutions to Hughes' model in $1d$ with zero Dirichlet conditions \eqref{eq:model} in cases for which a vacuum region is generated around the turning curve $x=\xi(t)$ for all times, i.e.~ \emph{no mass transfer occurs }through the turning point $\xi(t)$. At the particle level, this means that particles initially split into two sets moving towards distinct directions, with \emph{no change of direction} for all times.

The numerical simulations in section \ref{sec:numerics} supported this result, and show evidence that the proposed particle method is consistent also in cases in which particles do cross the turning point thus switching direction. The observation of such a phenomenon motivates further studies on the rigorous analytical study of the particle method, in particular to show (for the first time) an existence result for \eqref{eq:model} with \emph{large initial data}. Even the global-in-time well-posedness of the particle scheme is difficult in such cases. Indeed, even in some simple symmetric cases the continuation of the particle scheme after collision with the turning point cannot be prescribed by actually crossing the turning point itself. The authors will continue the study of this problem in a future work.

%For acknowledgements section, please don't number the section, please begin it with \section*{Acknowledgements}
\section*{Acknowledgments}
MDF is grateful to Peter A. Markowich (to whom this paper is dedicated) for having introduced him to the Hughes model for pedestrian movements in 2009. MDF and MDR are supported by the GNAMPA (Italian group of Analysis, Probability, and Applications) project \emph{Geometric and qualitative properties of solutions to elliptic and parabolic equations}. SF is supported by the GNAMPA (Italian group of Analysis, Probability, and Applications) project \emph{Analisi e stabilit\`a per modelli di equazioni alle derivate parziali nella matematica applicata}. GR was partially supported by ITN-ETN Marie Curie Actions ModCompShock - `Modelling and Computation of Shocks and Interfaces'. MDR acknowledges hospitality by the Gran Sasso Science Institute in L'Aquila in a period in which this manuscript was being written.

\medskip
% The data information below will be filled by AIMS editorial staff
Received xxxx 20xx; revised xxxx 20xx.
\medskip

\end{document}